\documentclass[a4paper,10pt,twoside]{amsart}
\usepackage{etex}
\usepackage{ifdraft}
\usepackage[utf8]{inputenc}
\usepackage[english,ngerman]{babel}
\usepackage[T1]{fontenc}
\usepackage{geometry}
\usepackage{dsfont}
\usepackage{easyeqn}
\usepackage[all,cmtip]{xy}
\usepackage{tikz}

\newcounter{fseqcount}[section]
\newcommand{\fslabel}[1]{\label{#1}}
\newcommand{\fseqlabel}[1]{\stepcounter{fseqcount}\label(\thesection.\arabic{fseqcount}){#1}}

\newtheorem*{thm_taut_impl_taut}{Theorem \ref{thm_0_taut_implies_most_p_taut}}

\newtheorem{thm}{Theorem}[section]
\newtheorem{prop}[thm]{Proposition}
\newtheorem{cor}[thm]{Corollary}
\newtheorem{leme}[thm]{Lemma}

\theoremstyle{remark}
\newtheorem{rem}[thm]{Remark}

\theoremstyle{definition}
\newtheorem{defin}[thm]{Definition}
\newtheorem{conj}[thm]{Conjecture}

\newcommand{\C}{\mathds{C}}
\newcommand{\FM}{\mathfrak{M}}
\newcommand{\N}{\mathds{N}}
\newcommand{\PR}{\mathds{P}}

\newcommand{\Aut}{\operatorname{Aut}}

\newcommand{\spec}{\operatorname{Spec}}
\newcommand{\MO}{\mathcal{O}}
\newcommand{\OH}[1]{\widehat{#1}}
\newcommand{\OS}[1]{\widetilde{#1}}
\newcommand{\kk}[1]{\overline{#1}}
\newcommand{\kdr}[2]{ {\Omega^1_{#1/#2}} }
\newcommand{\kd}[1]{ \kdr{#1}{k} }
\newcommand{\iso}{\cong}
\newcommand{\Ra}{\rightarrow}
\newcommand{\suml}{\sum\limits}
\newcommand{\prodl}{\prod\limits}
\newcommand{\liml}{\lim\limits}
\newcommand{\red}[1]{#1_{\operatorname{red}}}
\newcommand{\supp}{\operatorname{supp}}
\newcommand{\vp}{\varphi}

\newcommand{\I}{\mathcal{I}}
\newcommand{\del}{\partial}
\newcommand{\Hom}{\operatorname{Hom}}
\newcommand{\lRa}{\longrightarrow}

\newcommand{\id}{\operatorname{id}}
\newcommand{\SAut}{\mathcal{A}ut}
\newcommand{\AQ}{Q}
\newcommand{\SHom}{\mathcal{H}om}
\newcommand{\NOR}{\mathcal{N}}
\newcommand{\SL}{\mathcal{L}}
\newcommand{\maxl}{\max\limits}
\newcommand{\F}{\mathcal{F}}
\newcommand{\U}{\mathcal{U}}
\newcommand{\X}{\mathcal{X}}
\newcommand{\rank}{\operatorname{rank}}
\newcommand{\cdef}{\operatorname{Def}}
\newcommand{\Def}{\operatorname{Def}}
\newcommand{\Iso}{\operatorname{Iso}}

\newcommand{\CEQs}[1]{\operatorname{CEQ}(#1)}
\newcommand{\set}[2]{\{ #1\;|\;#2 \}}
\newcommand{\limindo}[1]{\lim\limits_{\underset{#1}{\longrightarrow}}}  

\newcommand{\condS}{(S)}

\begin{document}
\selectlanguage{english}

\title{On taut singularities in arbitrary characteristics}
\author{Felix Schüller}
\address{Mathematisches Institut\\ Heinrich-Heine-Universität\\ D-40225 Düsseldorf\\ Deutschland}
\email{schueller@math.uni-duesseldorf.de}
\date{\today}

\begin{abstract}
Over $\C$,  Henry Laufer classified all taut surface singularities. We adapt and extent his transcendental methods to
positive characteristic. With this we show that if a normal surface singularity is taut over $\C$, then the normal surface singularities
with isomorphic dual graph over algebraically closed fields of characteristic exponent $p>1$ are taut for all but finitely many $p$.
We conjecture that this is actually ``if and only if''.
\end{abstract}

\maketitle

\tableofcontents

\section*{Introduction}

Let $A$ be a normal, two-dimensional ring. One knows that $\spec(A) =S$ has at most isolated singularities and a
desingularization.
According to Laufers definition a normal two-dimensional singularity
is called \emph{taut} if every other normal two-dimensional singularity with isomorphic dual graph is already equivalent to $S$.
Recall that the dual graph $\Gamma$ for $S$ encodes the intersection of the regular components of the exceptional divisor of
the minimal good desingularization of $S$.

Over $\C$, Laufer classified all taut singularities by their dual graph \cite{MR0333238}. Over an arbitrary algebraically
closed field $k$ no such classification is known. The only known results in positive characteristic are the calculations of
Michael Artin on classes of ADE-singularities depending on the characteristic \cite{MR0450263}
and a recent proof of tautness for all Hirzebruch-Jung-singularities by Yongnam Lee and Noboru Nakayama \cite{ArX1103.5185}.

We follow Laufers approach over $\C$ (which uses transcendental methods) and prove many of his results for schemes over
arbitrary algebraically closed fields. With this we get the following theorem:
\begin{thm_taut_impl_taut}
Let $S_1$ be a normal two-dimensional singularity over $\C$ with dual graph $\Gamma$. For all primes $p$ let
$S_p$ be a singularity over an algebraically closed field of characteristic $p$ with dual graph $\Gamma$.
If $S_1$ is taut, then $S_p$ is taut for all but finitely many $p$.
\end{thm_taut_impl_taut}
We also conjecture that we have ``if and only if'' in this theorem. 
For the non-taut ADEs we can show that the number of non isomorphic singularities with a given dual graph $\Gamma$ is
$h^1(P,\Theta_P)+1$, where $P$ is the plumbing scheme for $\Gamma$.

\subsection*{Acknowledgments}

This article consists of a part of my Ph.D. thesis \cite{thesis_fschueller}. I like to thank my advisor Stefan Schröer 
for suggesting this interesting topic and for many helpful discussions, comments and suggestions during my work.
I would also like to thank Christian Liedtke, Phlipp Gross, Sasa Novakovic and Holger Partsch  for many helpful discussions.
Concerning the computations I like to thank Achim Schädle and Bertold Nöckel for the possibility to do some of them
on machines with a sufficient amount of memory.

\section{Tautness and cycles supported on the exceptional locus}

First we need to fix some notation and recall a few definitions. Whenever we write $k$ or $p$ without further specifications,
it is an arbitrary algebraically closed field, and $p$ is its characteristic exponent.

We say $S$ is a \emph{normal two-dimensional singularity}
if $S$ is the spectrum of a complete, normal, noetherian, local $k$-algebra $\MO_{S,s}$ with closed point $s$,
residue field $k$ and $\dim(S) =2$. Thus $\MO_{S,s}=k[ [ x_1,\dots,x_r]]/\mathfrak{a}$.

An \emph{algebraization} of $S$ is a noetherian, normal, local $k$-algebra $A$ of finite type with $\MO_{S,s} \iso \OH{A}$.
Theorem 4.7 of \cite{MR0262237} guaranties the existence of an algebraization for every normal two-dimensional singularity.
As usual for the classification of singularities, we only work with spectra of complete local rings. But for technical reasons
at some points we need the algebraization, because we need the desingularization to be smooth over $k$.

For a  normal two-dimensional singularity we have always a good desingularization. That is a desingularization such that
the integral components of the  exceptional divisor are regular and intersect transversally with no three distinct
components meet at one point. For all this good desingularizations exist a minimal one, that is one such that every other
good desingularization factors through it (take the one with the smallest number of components).
We call every $Z = \suml_{l=1}^n n_l E_l$ a \emph{cycle supported on the exceptional locus} if $\red{Z}$ is the reduction of the
exceptional locus of a minimal good desingularization of a normal two-dimensional singularity.
Later we also need the definition of the dual graph for these cycles, also we need more decorations as usually used:
\begin{defin}\fslabel{def_dual_graph}
Let $Z = \suml_{l=1}^n n_lE_l$ be a closed, 1-dimensional subscheme of a regular, two-dimensional scheme,
such that $Z$ is projective over $k$ and the $E_l$ are regular.
The \emph{dual graph} $\Gamma_Z$ of $Z$ is the following graph with multiple edges but without loops:
\begin{itemize}
     \item For each $E_l$ we have a vertex $v_l$.
     \item For $i \not= l$ we have $E_l \cdot E_i$ edges $e_{l,i}^j$ between $v_l$ and $v_i$.
     \item Each vertex $v_l$ is decorated by three weights: the arithmetic genus $p_a(E_l)$, the multiplicity $n_l$ and
	  the self-intersection $E_l^2$.
\end{itemize}
We say that two dual graphs are isomorphic if we have a bijection $\vp_v$ between the sets of vertices respecting the decorations
and a bijection $\vp_e$ between the sets of edges such that $e_{l,i}^j$ is mapped to an edge connecting $\vp_v(v_l)$ and $\vp_v(v_i)$.
\end{defin}
Let $E$ be the reduction of the exceptional divisor of the minimal good desingularization of a normal two-dimensional singularity $S$.
Then $E$ fulfils the assumptions of the previous definition, and we call $\Gamma_E$ the \emph{dual graph for $S$},
or we say $S$ is a \emph{$\Gamma_E$-singularity}.
$S$ is called \emph{taut} if $S$ is isomorphic to any other normal two-dimensional singularity with isomorphic dual graph.

The following lemma gives a criterion for two normal two-dimensional singularities to be isomorphic in term of direct systems
of cycles supported on the exceptional locus. It seems to be well-know, but we found no references for it, so we give a short proof.
\begin{leme}\fslabel{lem_equal_if_div_thickings}
Let $S_i$ be two normal two-dimensional singularities with minimal good desingularizations $f_i: X_i \Ra S_i$
and let $E_{i,l}$ be the integral components of the exceptional divisors.
Further let $(n_{1,j},. . ., n_{n,j})_{j \in \N}$ be a sequence
with $n_{l,j+1}\ge n_{l,j}$ and $\liml_{j \Ra \infty} n_{l,j}=\infty$ for all $l$.
Then $S_1$ is isomorphic to $S_2$ if and only if we have an isomorphism of direct systems
\[Z_{1,j} = \suml_{l=1}^n n_{l,j}E_{1,l} \iso \suml_{l=1}^n n_{l,j}E_{2,l}= Z_{2,j}\]
of schemes.
\end{leme}
\begin{proof}
First, if $S_1$ and $S_2$ are isomorphic, the direct systems of schemes
$X_i \otimes \MO_{S_i,s_i}/m_{s_i}^{l+1}$ ($i=1,2$, $l \ge 0$) are isomorphic. If on the other hand those systems are isomorphic,
we get an isomorphism of $S_1$ and $S_2$ using the theorem on formal functions and the normality of $\MO_{S_i,s_i}$.

Now the $X_i$ are noetherian and thus for every $\suml_{l=1}^nn_{l,j} E_{i,l}$ we find $r$, $r'$ and $j'$ such that
\[X_i \otimes \MO_{S_i,s_i}/m_{s_i}^{r} \subset \suml_{l=1}^nn_{l,j} E_{i,l} \subset X_i \otimes \MO_{S_i,s_i}/m_{s_i}^{r'} \subset
\suml_{l=1}^nn_{l,j'} E_{i,l}.\]
This chain shows, that the systems $X_i \otimes \MO_{S_i,s_i}/m_{s_i}^{l+1}$ are isomorphic
iff the systems $Z_{i,j}$ are, thus we get the first claim.
%
\end{proof}
We now want to use this to decide whether a singularity is taut. For this we need some additional notation:
Let $Z = \suml_{l=1}^n n_lE_l$ on $X$ and $Z' = \suml_{l=1}^{n'} n_l'E'_l$ on $X'$ be two closed, one-dimensional subscheme of 
regular, two-dimensional schemes, such that $Z$ and $Z'$ are projective over $k$ and the $E_l$, $E'_l$ are regular.
We say that $Z$ and $Z'$ are \emph{combinatorially equivalent} if their dual graphs are isomorphic.

We say that $Z$ is \emph{defined by its dual graph} if every $Z'$ combinatorially equivalent to $Z$
is already isomorphic to $Z$ as scheme.

By $\CEQs{Z}$ we denote the set of all tuple $(Z',X')$ where $X'$ is a regular,
two-dimensional $k$-scheme and $Z'\subset X'$ is combinatorially equivalent to $Z$, divided by the equivalence relation given by
$ (Z',X') \sim (Z'',X'')$ iff $Z'$ is isomorphic to $Z''$ as $k$-schemes.
Then $Z$ is defined by its dual graph if and only if $\CEQs{Z} =\{[(Z,X)]\}$.

With this definition Lemma \ref{lem_equal_if_div_thickings} shows that a normal two-dimensional singularity is taut,
if all $Z_j$ are defined by their dual graphs.
The reverse of this is more delicate. Suppose we have a $Z_j$ and find a $Z'$ combinatorial equivalent,
but not isomorphic. Then we get a whole system of schemes, combinatorial equivalent, but not isomorphic.
We know (by definition) that $Z'$ is embedded in a regular, two-dimensional scheme
$X'$. But it is well known, that if we contract $Z' \subset X'$ we may only get an algebraic space.

We now want to show, that we can contract $Z' \subset X'$ as a scheme, if we modify $X'$ away from $Z'$:
\begin{leme}\fslabel{lem_contr_scheme}
Let $Z = \suml_{l=1}^n n_lE_l$ be a closed, one-dimensional subscheme of a regular, two-dimensional scheme $X$,
such that $Z$ is projective over $k$ and the $E_l$ are integral. If $\red{Z}$ satisfies the conditions
of the exceptional divisor of a minimal good desingularization, then $Z$ is the exceptional divisor of a minimal good
desingularization, that is, there exists a normal two-dimensional singularity $S'$
with minimal good desingularization $f': X' \Ra S'$ and an embedding $\iota: Z \Ra X'$ with $f'(\iota(Z))=s'$.
\end{leme}
\begin{proof}
By Corollary (6.12) of \cite{MR0260747}  we have a contraction $f: X \Ra S$ of $Z$ with $S$ an algebraic space and $s=f(Z)$.
Then by Theorem II 6.4 of \cite{MR0302647} we have an affine scheme $U$ and an étale map $U \Ra S$
such that the embedding $s \Ra S$ factors $s \Ra U \Ra S$.
We may assume $U$ to be normal. We take the fibre product of algebraic spaces
$X' =X \times_{S} \spec(\OH{\MO}_{U,s})$. Now $S'= \spec(\OH{\MO}_{U,s})$ is a scheme, and by Proposition II 1.7 of \cite{MR0302647}
we know that the fibre product of two schemes over an algebraic space is a scheme, so $X'$ is a scheme. Let $s'$ be the closed point of
$S'$. Then we know that $S'$ is a normal two-dimensional singularity and $X'$ is regular.
Because the reduction of the exceptional fibre of $f': X' \Ra S'$ is
$\red{Z}$, we know that $f'$ is the minimal good desingularization of $S'$.

It remains to prove the existence of $\iota$. First we remark that by the same argumentation as above we get
$Z \subset X\otimes \spec(\OH{\MO}_{U,s}/m^{i+1})$ for an $i$ large enough.
But by definition we have $\OH{\MO}_{U,s}/m^{i+1} = \MO_{S',s'}/m_s^{i+1}$, and so the associativity of the fibre product gives us
\[Z \subset X\otimes \spec(\OH{\MO}_{U,s}/m^{i+1}) \iso X' \otimes \spec(\MO_{S',s'}/m_s^{i+1})\]
and this gives the wanted $\iota: Z \Ra X' $.
\end{proof}
The $Z_{i,j}$ of Lemma \ref{lem_equal_if_div_thickings} may also be calculated on the minimal good desingularization of
any algebraization of $S$. Thus for tautness the lemma can be restated as:
\begin{prop}\fslabel{prop_taut_iff_all_comb_equiv_iso}
Let $S$ be a normal two-dimensional singularity. Let $f$ be the minimal good desingularization of $S$ or of any algebraization of
$S$ and let $E_l$ be the $n$ integral components of its exceptional divisor.
Let $(n_{1,j},. . ., n_{n,j})_{j \in \N}$ be a sequence with $n_{l,j+1}\ge n_{l,j}$ and
$\liml_{j \Ra \infty} n_{l,j}=\infty$ for all $l$. We set $Z_j = \suml_{l=1}^n n_{l,j} E_l$.
Then $S$ is taut if, and only if for all $j$ the $Z_j$ are defined by their dual graphs.
\end{prop}
Our next goal is to give a necessary condition on the structure of the dual graph for a normal two-dimensional singularity to be taut.
For this we need to discuss some special cycles supported on the exceptional locus.
Recall that for a normal two-dimensional singularity by \cite{MR0199191}, Page 131
we have the \emph{fundamental cycle}, that is the smallest divisor $Z$ on $X$ with $\supp(Z)=\supp(E)$ and $Z\cdot E_i \le 0$
for all $i$. By Lemma 4.10 of \cite{MR0320365} we also know, that we find at least one cycle with strict inequality for all $i$.
We call such a cycle an \emph{anti-ample cycle for $S$}, because by \cite{MR0276239}, Theorem 12.1 (iii) the negative of it is ample.
Because the coefficients of an anti-ample cycle only depend on the dual graph $\Gamma$ of $S$, we may also speak of an
\emph{anti-ample cycle for $\Gamma$}.
If we have $p>1$ we sometimes need an anti-ample cycle with coefficients prime to $p$.
We get the existence of such anti-ample cycles from the following lemma:
\begin{leme}\fslabel{lem_prim_anti_ample_cycle}
Let $S$ be a normal two-dimensional singularity, then we always have an anti-ample cycle
$\OS{Z} = \suml_{l=1}^n n_l E_l$ for $S$ such that $\gcd(n_l,p)=1$ for all $l$.
\end{leme}
\begin{proof}
For $p=1$ there is nothing to show. For $p>1$ let $\OS{Z}'$ be any anti-ample cycle for $S$ and let
$t=\max\limits_{i} \{E_i \cdot (E_1+\dots +\OH{E_{i}}+\dots +E_n)\}$. 
We write $(t+1)\OS{Z}' = \suml_{l=1}^n n'_l E_l$, and define $n_l$ by $n_l = n'_l+1$ if
$p | n'_l$ and $n_l=n'_l$ else and set $\OS{Z} =\suml_{l=1}^n n_l E_l$.
A calculation shows that $\OS{Z}$ is anti-ample.
\end{proof}
Also we need that for ever suitable dual graph $\Gamma$ and every appropriate choice of curves $E_l$ we find a
$\Gamma$-singularity $S$ such that the exceptional locus consists of the $E_l$. The first step for this is the following proposition:
\begin{prop}\fslabel{prop_real_exist_for_graph}
For any connected dual graph $\Gamma$ with negative definite $(E_i \cdot E_j)$, and any $n$ smooth, one-dimensional schemes $E_l$
with $p_a(E_l)$ as in $\Gamma$, we can embed
$Z=\suml_{l=1}^n n_l E_l$ into a smooth, two-dimensional scheme $X$ such that the dual graph of $Z$ is $\Gamma$.
\end{prop}
\begin{proof}
First we note that it suffices to prove the proposition for one chosen $n$-tuple $(\OS{n}_1, \dots, \OS{n}_n)$ of natural numbers,
which may differ from the $n_l$ of $\Gamma$.
The main difficulty now is not to find a $X$ into which $Z$ embeds, but to find a $X$ such that ifor all $i$ the
$E_i^2$ equals to the self-intersection given by $\Gamma$.
To find this we use the following fact: Suppose we have a closed, one-dimensional subscheme
$Z'=\suml_{l=1} n'_lE'_l$ of a smooth, two-dimensional scheme $X'$ such that $Z'$ is the fibre of a map from $X'$ to a smooth,
one-dimensional scheme. Then we have $0 = Z' \cdot E'_i$ for all $E'_i$ 
and thus the $ (E'_i)^2 $ are controlled by the $n'_l$ and the $E'_i \cdot  E'_l$. 

We now choose $(\OS{n}_1, \dots, \OS{n}_n)$ such that $\OS{Z}=\suml_{l=1} \OS{n}_lE_l$ is an anti-ample cycle.
Then we construct a new $Z'$ from $\OS{Z}$ as follows: At every $E_i$ we choose $r_i=-\OS{Z} \cdot E_i$ points which are smooth in
$E$ and we glue additional smooth, one-dimensional schemes $E_{i,j}$ transversally to them such that $E_{i,j}$ only intersects
with $E_i$.

By construction, this $Z'$ fulfils the assumptions of Proposition 4.2 of \cite{MR0357406}, and can thus be embedded as a fibre 
into a smooth, two-dimensional scheme $X$. This induces the wanted embedding of $Z$.
\end{proof}
Now, together with Lemma \ref{lem_contr_scheme} we get:
\begin{cor}\fslabel{cor_sing_exist_for_graph}
For any connected dual graph $\Gamma$, with negative definite $(E_i \cdot E_j)$
and any $n$ smooth, one-dimensional schemes $E_l$ with $p_a(E_l)$ as in $\Gamma$, we have normal two-dimensional singularity $S$
with desingularization $f:X \Ra S$ such that $\Gamma$ is the dual graph of $\suml_{l=1}^n n_l E_l$ on the exceptional locus.
\end{cor}
Now suppose we have a singularity with $n=1$ and $p_a(E_1)=1$.
Than $E_1$ is an elliptic curve, and the isomorphism type is described by the $j$-invariant. But the $j$-invariant is not encoded in
the dual graph. So if we take two elliptic curves with different $j$-invariants, then with Proposition \ref{prop_real_exist_for_graph}
we can embed both curves with a given negative self-intersection into smooth surfaces. Then this curves are combinatorially equivalent,
but not isomorphic. This implies that if we contract
these curves, the resulting singularities are not isomorphic, but have isomorph dual graphs, and thus are not taut.

This example generalises in the following way to the wanted necessary condition on $\Gamma$ for $S$ to be taut:
\begin{leme}\fslabel{lem_necessary_cond_for_taut}
Let $S$ be a normal two-dimensional singularity, $f:X \Ra S$ its minimal good desingularization, and $E_i$ the integral
components of the exceptional locus.
If $S$ is taut, then we have $p_a(E_i)=0$ and each $E_i$ intersects at most 3 others.
\end{leme}
\begin{proof}
First suppose by contradiction that we have an $i$ with $p_a(E_i)>0$. The case $p_a(E_i)=1$ is just a direct
generalisation of the example above. The general case for $g=p_a(E_i)>1$ follows analogously using the scheme $\FM_g$.

So we have necessarily $p_a(E_i)=0$ for all $i$. Because $k$ is algebraically closed, this is equivalent to $E_i \iso \PR^1_k$.

Now assume we have an $E_i$ which intersects with $4$ others.
We may assume that $3$ of the $4$ other components intersecting $E_i$ intersect at $0$, $1$ and $\infty$.
Now we take $E$ and $E'$, such that the 4-th component intersects $E_i$ at different points. Then we can again embed $E$ and $E'$,
and their dual graphs are isomorphic, but $E$ is not isomorphic to $E'$, and so $S$ is not taut.
\end{proof}
\begin{defin}
We say a  normal two-dimensional singularity is \emph{potentially taut} if its minimal good desingularization fulfils
the conclusion of the previous lemma.
We call a dual graph $\Gamma$ \emph{potentially taut} if it is the dual graph of a cycle supported on the exceptional locus of a
potentially taut normal two-dimensional singularity.
\end{defin}

\section{Extending isomorphisms of exceptional schemes}

The goal of this section is to show that, with the notation of Proposition \ref{prop_taut_iff_all_comb_equiv_iso}, a normal
two-dimensional singularity $S$ is taut iff for one $j_0 \gg 0$ the $Z_{j_0}$ 
is defined by its combinatorial data.

The main tool for this is a obstruction-theory, which in the analytic category, was developed by Grauert
(\cite{MR0137127}), Laufer (\cite{MR0320365}) and Tjurina (\cite{MR0246880}). It gives a criterion whether an isomorphism between
cycles $Z_{j}$ and $\OS{Z}_{j}$ can be extended to one between $Z_{j} +E_l$ and $ \OS{Z}_{j} + \OS{E}_l$. Laufer showed that,
starting with a special $Z_{j_0}$, one always finds a sequence of $E_l$ to add, such that the obstruction against the extending
to every $Z_{j}$ bigger is trivial.

Most of the results we need stay true in the algebraic category, the proofs only need small modifications, so we omit those.
The only thing that is more difficult in the algebraic category, is to prove that the isomorphism can be extended locally.

For this section we need to replace $S$ with an algebraization.
Then we know that the exceptional divisor is a local complete intersection in a smooth, two-dimensional $k$-scheme.

So, for this section, let $B=\suml_{l=1}^n n_l B_l$ be a divisor on a smooth $k$-surface $X$, with the $B_l$ regular,
and the singularities of $\red{B}$ are transversal intersections of at most two components. Further, let
$C= \suml_{l=1}^n n'_l B_l$ with $0 < n'_l \le n_l$.
Also we need analogous  $\OS{B}=\suml_{l=1}^n n_l \OS{B}_l \subset \OS{X}$ and $\OS{C}=\suml_{l=1}^n n'_l \OS{B}_l$.

First we want to show, that locally, we always are able to extent an isomorphism between $C$ and $\OS{C}$ to one between
$B$ and $\OS{B}$:
\begin{prop}\fslabel{prop_local_thicken}
Suppose that $\vp: C \Ra \OS{C}$ is the isomorphism.
Then for every $x \in B$ there exists an open $x \in U_x \subset B$ and an isomorphism $\psi: U_x \Ra \vp(U_x)$ such that
$\psi|_{C\cap U_x } = \vp|_{U_x}$.
\end{prop}
\begin{proof}
Let $x \in B_l$. Then we can find a regular $k$-algebra $A$ of finite type and $f,g \in A$ such that
$\spec(A/(f^{n_l}g^{n_j})) = U_x \subset B$. (Take $f$ as a local equation for $B_l$ in $X$ and $g$ one for $B_j$,
if $x \in B_l \cap B_j$ and $g=1$ if $x$ is a regular point of $\red{B}$.)
If we do this also around $\vp(x)$, we get, by abuse of notation,
$\vp: A/(f^{n'_l}g^{n'_j}) \Ra \OS{A}/(\OS{f}^{n_l'}\OS{g}^{n_j'})$. Further, if $n'_l >1$ or $g\not =1$, then
we may choose $\OS{f}$ and $\OS{g}$ in such a way that the $\vp$ maps the residue class of $f$ to the residue class of $\OS{f}$,
and the one of $g$ to the one of $\OS{g}$.

Now it suffices to show the proposition for the case $n_l = n_l'+1$ and $n_j=n_j'$.

First we do the cases that $x$ is no regular point of $\red{B}$, that is $g\not=1$.

Let $\Phi: A \Ra \OS{A}/(\OS{f}^{n_l'}\OS{g}^{n_j})$ the map we get by composing with $\vp$.
Since $A$ is regular and of finite type, and $k= \kk{k}$, $\spec(A)$ is smooth, and we get a map
$\Psi: A \Ra \OS{A}/(\OS{f}^{n_l'+1}\OS{g}^{n_j})$ from $\Phi$ by the infinitesimal lifting property.
First we want to show that we can choose $\Psi$ such that it maps $f$ and $g$
to the residue class of $\OS{f}$ and $\OS{g}$. The ideal $(\OS{f}^{n_l'}\OS{g}^{n_j})/(\OS{f}^{n_l'+1}\OS{g}^{n_j})$
is an $ \OS{A}/(\OS{f}^{n_l'}\OS{g}^{n_j})$-module, and so using $\Phi$ also an $A$-module.
Let now $\del$ be any $k$-derivation from $A$ to $(\OS{f}^{n_l'}\OS{g}^{n_j})/(\OS{f}^{n_l'+1}\OS{g}^{n_j})$,
that is a $k$-linear map fulfilling the Leibniz rule $\del(ab)=\Phi(a)\del(b)+\Phi(b)\del(a)$ for all $a,b \in A$.
Then for $\Psi'=\Psi+\del$ we have $\Psi'(ab)=\Psi'(a)\Psi'(b)$ by a straight-forward calculation using the Leibniz rule
and $\Psi \cdot \del=\Phi \cdot \del$ in $(\OS{f}^{n_l'}\OS{g}^{n_j})/(\OS{f}^{n_l'+1}\OS{g}^{n_j})$.

So $\Psi'$ is also a lifting of $\Phi$. Now $\Psi(f)-\OS{f}$ and $\Psi(g)-\OS{g}$ in are in the kernel of $\pi$ and thus
in $(\OS{f}^{n_l'}\OS{g}^{n_j})/(\OS{f}^{n_l'+1}\OS{g}^{n_j})$, so we can choose $\Psi$ in the described way if we find a
derivation $\del$ such that $\del(f)=-(\Psi(f)-\OS{f})$ and $\del(g)=-(\Psi(g)-\OS{g})$. Now we use the standard identification
between derivations and elements of $\Hom_{A}(\kd{A}, (\OS{f}^{n_l'}\OS{g}^{n_j})/(\OS{f}^{n_l'+1}\OS{g}^{n_j}))$. If $y$ is is
the singular point of $\spec(A/(fg))$, 
then $\kd{A_x}$ is generated
by $d_xf$ and $d_xg$, but $\kd{A}$ is finitely generated and quasi-coherent,
so maybe after shrinking $U_x$ we may assume that $df$ and $dg$ generate $\kd{A}$.
This shows the existence of a derivation $\del$ with $\del(f)=-(\Psi(f)-\OS{f})$ and $\del(g)=-(\Psi(g)-\OS{g})$ and thus we
can assume $\Psi(f)=\OS{f}$ and $\Psi(g)=\OS{g}$ in $\OS{A}/(\OS{f}^{n_l'+1}\OS{g}^{n_j})$.

Now we have $\Psi(f^{n_l'+1}g^{n_j})=0$ and we get $\psi'$ by the universal property
of the kernel. If we do the same for $\vp^{-1}$ and $\OS{A}$ we get
$\OS{\psi}: \OS{A}/(\OS{f}^{n_l'+1}\OS{g}^{n_j}) \Ra A/(f^{n_l'+1}g^{n_j})$.

By construction we get for all $a \in A/(f^{n_l'+1}g^{n_j})$ and all
$b \in \OS{A}/(\OS{f}^{n_l'+1}\OS{g}^{n_j})$:
\[\OS{\psi} \circ \psi'(a) = a + f^{n_l'}g^{n_j}\del(a) \text{ and }
\psi' \circ \OS{\psi}(b) = b + \OS{f}^{n_l'}\OS{g}^{n_j}\OS{\del}(b)\]
Now we set for $a \in A/(f^{n_l'+1}g^{n_j})$:
\[\xi(a) = a- f^{n_l'}g^{n_j} \del(a) \text{ and } \psi = \psi' \circ \xi\]
Then a strict forward calculation shows, that $\psi$ is the isomorphism we need.

For the remaining case, $g=1$. If $n'_l\ge 2$ the previous argumentation holds also in this case,
we only have to replace $dg$ with some $dg'$ such that $\kd{A}$ are generated by $df$ and $dg'$ at one place.
So we have only to do the case $n'_l=1$.
But in this case this follows because $\spec(A/(f))$ is smooth.
\end{proof}
Now we go back to the global situation. Here we may assume that
$C \subset B \subset X$ are given by ideal sheaves
$\I_B = \prodl_{l=1}^n \I_l^{n_l}$ and $\I_C = \prodl_{l=1}^n \I_l^{n'_l}$.
That is, we have an exact sequence
\[ 0 \lRa \I_C/\I_B \lRa \MO_X/\I_B \lRa \MO_X/\I_C \lRa 0\] 
%

Now we want to construct a sheaf classifying automorphism $\alpha$ of $B$ which are the identity on $C$, that is
the sections of this sheaf are not automorphisms of $B$, but of the $\MO_X$-algebra $\MO_X/\I_B$. So by the
well known contravariant correspondence between automorphisms of $B$ and automorphisms of $\MO_X/\I_B$, the
sections of this sheaf are the opposite group to the group of automorphisms of $B$.
The identity condition restricted to $C$ then translates to the commutativity of the following
diagram:
\begin{EQ}\fseqlabel{dia_def_aut}
\begin{gathered}
\xymatrix{
0 \ar[r] & \I_C/\I_B \ar[r] \ar[d]^{\alpha|_{\I_C/\I_B}} & \MO_X/\I_B \ar[r] \ar[d]^{\alpha} & \MO_X/\I_C \ar[r] \ar[d]^{\id} & 0\\
0 \ar[r] & \I_C/\I_B \ar[r]  & \MO_X/\I_B \ar[r]  & \MO_X/\I_C \ar[r] & 0 }
\end{gathered}
\end{EQ}
From this commutativity we get, that $\alpha$ maps $\I_C/\I_B$ necessarily to $\I_C/\I_B$ and
using the snake-lemma we get that the restriction $\alpha|_{\I_C/\I_B}$ must be already surjective.

Now we define the pre-sheaf $\SAut_C(B)$ whose sections for an open $U \subset B$ are defined
as the set of all isomorphisms $\alpha: \Gamma(U,\MO_B|_U) \Ra \Gamma(U,\MO_B|_U)$ such that $\alpha$ is the identity on the set
$U$ and for all $x \in U$ we have $\alpha_x( (\I_C/\I_B)_x) = (\I_C/\I_B)_x$ and $\alpha$ induces the identity on $\MO_{C,x}$.
Then $\Gamma(U,\SAut_C(B))$ together with the composition is a group.
By the discussion above the automorphism making the diagram above commutative are exactly the global sections of $\SAut_C(B)$.
Also the pre-sheaf $\SAut_C(B)$ is a sheaf.

Now the proof of Theorem 6.6.~of \cite{MR0320365} applies without change in our situation, so we get:
\begin{prop}\fslabel{prop_obstr_lift_iso}
Let $\psi: C \Ra \OS{C}$ be an isomorphism and assume that we can extend $\psi$ locally.
Then the local extensions determine a class $o \in \check{H}^1(B, \SAut_C(B))$, and $o = *$ if and only if
we can glue the local extensions to a global isomorphism $\Psi: B \Ra \OS{B}$.
\end{prop}
The other direction is also true: If two schemes become isomorphic after thickening, the are isomorphic.

Now, under some additional conditions, the pointed set $\check{H}^1(B, \SAut_C(B))$ is actually computable,
and is in most cases even a group.

The sheaf $\SAut_C(B)$ has a subsheaf $\SAut_{C,\I_C/\I_B}(B)$ of normal subgroups given by
\[\Gamma(U, \Aut_{C,\I_C/\I_B})=\set{\alpha \in \Gamma(U, \Aut_{C})}{ \alpha_x \text{ is the identity on }
(\I_C/\I_B)_x \forall x \in U}\]
and if we denote by $\AQ$ the quotient sheaf we get an exact sequence of sheaves of groups:
\begin{EQ}\fseqlabel{sh_sq_aut}
1 \lRa \SAut_{C,\I_C/\I_B}(B) \lRa \SAut_{C}(B) \lRa \AQ \Ra 1
\end{EQ}

Now as first condition, we assume $\I_C^2 \subset \I_B$, that is for every open $U \in X$ we have $(\I_C(U))^2 \subset \I_B(U)$ in
$\MO_X(U)$. Then like \cite{MR0320365}, Proposition 6.4, we can get an isomorphism
\[\lambda:  \SHom_{\MO_{B}}(\kdr{C}{k}, \I_C/\I_B) \lRa \SAut_{C,\I_C/\I_B}(B)\]

We further need the following condition:
\begin{defin}
If we say that \emph{$C \subset B$ fulfil condition \condS}, if we have exactly one $l_0$ with $n_{l_0}=n'_{l_0}+1$
and for all other $l$ we have $n_l=n'_l$.
\end{defin}
First we note, that if $C \subset B$ fulfil condition  \condS, then we have $\I_C^2 \subset \I_B$,
so we still have the isomorphism $\lambda$.
Following Laufers calculations in \cite{MR0320365} one gets:
\begin{cor}\fslabel{cor_relevant_h1}
Suppose $C \subset B$ fulfil condition \condS, then $\check{H}^1(B,\SAut_{C}(B))$ vanishes,
if the following cohomology groups respectively sets vanish:
\begin{itemize}
     \item If $n_{l_0}=2$: $H^1(B_{l_0},(\kdr{B_{l_0}}{k})^\vee \otimes_{\MO_{B_{l_0}}} \I_C/\I_B)$ and $\check{H}^1(B, \AQ)$.
\item If $n_{l_0}> 2$: $H^1(B_{l_0},(\kdr{B_{l_0}}{k})^\vee \otimes_{\MO_{B_{l_0}}} \I_C/\I_B)$ and
     $H^1(B_{l_0},\NOR_{B_{l_0}/X} \otimes_{\MO_{B_{l_0}}} \I_C/\I_B)$.
\end{itemize}
\end{cor}
Now by calculating the degree of the involved sheaves, Serre duality and that $\deg(\SL) < 0$ on an integral,
one-dimensional, proper $k$-scheme implies $h^0(Y,\SL) =0$ we get:
\begin{cor}\fslabel{cor_vanish_by_num}
Suppose $C \subset B$ fulfil condition \condS, then:
\[ H^1(B_{l_0},(\kdr{B_{l_0}}{k})^\vee \otimes_{\MO_{B_{l_0}}} \I_C/\I_B)=0 \text{ if }
2(2p_a(B_{l_0})-2) +B_{l_0} \cdot \suml_{l=1}^n n'_l B_l < 0.\]
If we additionally have $n'_{l_0}\ge 2$, then
\[H^1(B_{l_0},\NOR_{B_{l_0}/Y} \otimes_{\MO_{B_{l_0}}} \I_C/\I_B) =0 \text{ if }
 2p_a(B_{l_0})-2 - B_{l_0}\cdot B_{l_0} +B_{l_0} \cdot \suml_{l=1}^n n'_l B_l < 0.\]
\end{cor}
Combining the previous two corollaries we get the following reformulation of Theorem 6.8 of \cite{MR0320365}:
\begin{prop}\fslabel{prop_iso_lifts_if_vanish}
Suppose $C \subset B$ fulfil condition \condS, and $n'_{l_0} \ge 2$, and if the conditions of Corollary \ref{cor_vanish_by_num}
are fulfilled, then the map $ \CEQs{C} \Ra \CEQs{B}$ mapping $[(C',X)]$ to the equivalence class of any extension of $C'$
is a well-defined bijection.
\end{prop}
Now we want to transfer the results above to the tautness of a normal two-dimensional singularity. For this fix one anti-ample cycle
$\OS{Z} = \suml_{l=1}^nr_l E_l$ for $S$. 
We construct $\OS{Z}$ stepwise, that is, let $r = \suml_{l=1}^n r_l$, we choose $\beta_0,...,\beta_{r-1} \in \{1,...,n\}$
as follows: $\OS{Z}_1=E_{\beta_0}$, $\OS{Z}_r = \OS{Z}$ and for all $i \in \{1,...,r-1\}$ we have
$\OS{Z}_{i+1} = \OS{Z}_i + E_{\beta_{i}}$. By construction, if we set $B=\OS{Z}_{i+1}$ and $C=\OS{Z}_{i}$ they fulfill condition
\condS, and we can use our calculations above.
We define
\[\tau = \maxl_{i \in \{1,...,r-1\}}(E_{\beta_i}\cdot \OS{Z}_i)\text{ and }
\lambda = \maxl_{l \in \{1,...,n\}}\{0, 2(2p_a(E_l) -2), 2p_a(E_l) -2- E_l^2\}.\]
Proposition \ref{prop_iso_lifts_if_vanish} now gives us the following, which is a reformulation of Theorem 6.9 of \cite{MR0320365}:
\begin{prop}\fslabel{prop_red_to_finite}
Let $\OS{Z}$, $\tau$ and $\lambda$ as above. If we have $\nu \ge \max\{\lambda + \tau +1,1\}$ and if at least one $n_l$  is equal
to $1$ additionally $\nu \ge 2$,
than we have a bijection \[\CEQs{\nu \OS{Z}}\lRa \CEQs{ (\nu+1) \OS{Z}}\]
\end{prop}
As a corollary we get the same result for an arbitrary $B \ge \nu \OS{Z}$:
\begin{cor}\fslabel{cor_osz_triv_ind_greater_triv}
If $B\ge \nu \OS{Z}$, then we have a bijection
\[\CEQs{\nu \OS{Z}} \lRa \CEQs{ B}\]
\end{cor}
\begin{proof}
There exist a $l \in \N$ with $B \le (\nu+l)\OS{Z}$. Now 
the maps
\[\CEQs{\nu \OS{Z}} \lRa \CEQs{ B} \lRa \CEQs{(\nu+l) \OS{Z}},\]
are injective, but by the previous proposition the composition is also bijective, so the first map is already bijective.
\end{proof}
Now our $\nu$ still depends on the choice of the $\beta_i$, but there are only finitely many choices, so we have a minimal $\tau$,
which we call $\tau_{min}$. Then we define:
\begin{defin}\fslabel{def_sign_multi}
Let $S$ be a normal two-dimensional singularity and $\OS{Z}=\suml_{l=1}^n n_{l} E_l$ an anti-ample cycle for $S$.
The \emph{significant multiplicity of $\OS{Z}$} is the smallest integer $\nu$ such that
$\nu \ge \lambda + \tau_{min} +1$, and $\gcd(p,\nu)=1$; if at least one $n_l$ is equal to $1$, then we furthermore demand $\nu \ge 2$.
\end{defin}
Note that the condition $\gcd(p,\nu)=1$ is not necessary for the results of this section, but later it simplifies the formulations.
By definition the $\nu$ only depends on the dual graph of $\OS{Z}$.
Now we can simply take one order $\beta_0,\dots,\beta_{r-1}$ such that $\tau$ is minimal, and immediately get
the following corollary of Proposition \ref{prop_red_to_finite} respectively Corollary \ref{cor_osz_triv_ind_greater_triv}:
\begin{cor}\fslabel{cor_induction_neqs}
If $\nu$ is the significant multiplicity of $\OS{Z}$ and $\CEQs{\nu \OS{Z}} = \{ [(\nu \OS{Z},X)]\}$,
then for all $B=\suml_{l=1}^l n_lE_l$ we have $\CEQs{B} = \{ [(B,X)]\}$
\end{cor}
Finally, the translation back to singularities is the wanted result:
\begin{cor}\fslabel{cor_taut_if_neqs_trivial}
Let $S$ be a normal two-dimensional singularity, $\OS{Z}$ an anti-ample cycle for $S$ and $\OS{\nu}$ its significant
multiplicity. Then $S$ is taut if and only if 
\[\CEQs{j_0 \OS{Z}} = \{[(j_0 \OS{Z},X)]\}\]
for one $j_0 \ge \OS{\nu}$.
\end{cor}
\begin{proof}
Let $\nu$ be as in Proposition \ref{prop_red_to_finite}.
By Corollary \ref{cor_osz_triv_ind_greater_triv} we have $\CEQs{j_0 \OS{Z}} = \{[(j_0 \OS{Z},X)]\}$ for one $j_0 \ge \OS{\nu}\ge \nu$ if
and only if $\CEQs{\nu \OS{Z}} = \{[(\nu \OS{Z},X)]\}$. So the corollary is an immediate consequence of Proposition
\ref{prop_taut_iff_all_comb_equiv_iso} and Corollary \ref{cor_induction_neqs}.
\end{proof}

\section{The plumbing scheme}\fslabel{sect_first_steps}

The last corollary of the previous section reduces the tautness of a normal two-dimensional singularity $S$ 
to the triviality of $\CEQs{\nu \OS{Z}}$. For $n=1$ and $n=2$ this can easily be calculated with the techniques of the
previous section. But already for $n=3$ some of the obstruction groups given in Corollary \ref{cor_relevant_h1} are not trivial.
We use an other approach to get the triviality of $\CEQs{\nu \OS{Z}}$. First by Lemma \ref{lem_necessary_cond_for_taut} we may assume
that $S$ is potentially taut.
Following Laufers idea, we construct a special scheme $P$ with dual graph $\Gamma_{\nu \OS{Z}}$,
the so called plumbing scheme. We then show that $\CEQs{\nu \OS{Z}}$ is trivial if $H^1(P,\SHom_{\MO_P}(\kd{P}, \MO_P)) =0$.

For this let $\Gamma$ be any potentially taut dual graph with vertices $E_l$ and multiplies $n_l$ and self-intersection $-\nu_l$.
Let $t_l$ be the number of edges at $E_l$. From the potentially tautness we have $t_l < 4$.
For every vertex $E_l$ we construct a scheme $W_l$ and glue them together to get $P$. This $W_l$ consists of a $(n_l-1)$-times
thickened $\PR^1_k$ for $E_l$ and for ever $E_{j_i}$
($1 \le i \le t_l$) which is directly connected by an edge to $E_l$ we add a $(n_{j_i}-1)$-times thickened affine arm at
$0$, $\infty$ or $1$. That is, for the reduction we get the following picture:
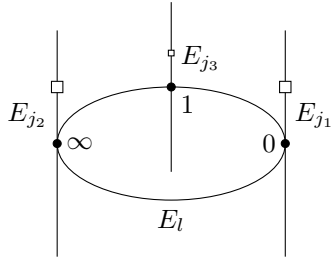
\[\begin{tikzpicture}[scale=1.5]
\draw (0,0) ellipse (1 and .5);	
\draw (1,-1) -- (1,.45);		
\draw (.95,.45) rectangle (1.05,.55);
\draw (1,.55) -- (1,1);
\draw (-1,-1) -- (-1,.45);	
\draw (-1.05,.45) rectangle (-.95,.55);
\draw (-1,.55) -- (-1,1);
\draw (0,-.25) -- (0,.775);	
\draw (-.025,.775) rectangle (.025,.825);
\draw (0,.825) -- (0,1.25);
\filldraw [black]
	(1,0) circle (1pt) node[left] {$0$}		
	(0,.5) circle (1pt) node[below right] {$1$}	
	(-1,0) circle (1pt) node[right] {$\infty$};	
\draw (0,-.5) node[below] {$E_l$};	
\draw (1,.25) node[right] {$E_{j_1}$};	
\draw (-1,.25) node[left] {$E_{j_2}$};	
\draw (0,.75) node[right] {$E_{j_3}$};	
\end{tikzpicture}\]
The affine charts of $W_l$ are given as the spectra of $R_{l,i} = k[x_{l,i},y_{l,i},(y_{l,i} -1)^{-1}]/\mathrm{Rel}_{l,i}$ ($i=0,1$)
Where $\mathrm{Rel}_{l,i}$ is given by the fact that $E_l=V(y_{l,0}^{n_l}) \cup V(y_{l,1}^{n_l})$ and $E_{j_1} = V(x_{l,0}^{n_{j_1}})$.
If we have $t_l \ge 2 $ then $E_{j_2} = V(x_{l,1}^{n_{j_2}})$ and finally, if $t_l=3$, then
$E_{j_3}=V( (x_{l,0}-1)^{n_{j_3}}) \cup V( (x_{l,1}-1)^{n_{j_3}})$. 
Those two charts are glued via $x_{l,0}x_{l,1} -1$ and $y_{l,0} - x_{l,1}^{\nu_l}y_{l,1}$.

We need to invert the $y_{l,i} -1$ because there we may glue an other $E_j$ to one of the $E_{j_i}$.
For all practical calculations we need later, this can be ignored, because inverting this elements is just a localization,
and those commute for example with taking Kähler differentials.

Now we want to glue the $W_l$. For this it suffices to give the glueing for $W_l$ and $W_{j_i}$.
For this let $j=j_i$, $\OS{x}_{l,i_l}=x_{l,i_l}-1$ if $i=3$ in $W_l$ and $\OS{x}_{l,i_l}=x_{l,i_l}$ else, and analogously
for $\OS{x}_{j,i_j}$ depending on $l=3$ in $W_j$. We set
$\OS{A}_{lj} = k[\OS{x}_{l,i_l},y_{l,i_l},\OS{x}_{j,i_j},y_{j,i_j},(y_{l,i_l}-1)^{-1},(y_{j,i_j}-1)^{-1}]$
and glue via:
\begin{EQ}\fseqlabel{eq_charts_glueing_P}
\OS{R}_{lj} = \OS{A}_{lj}/(\OS{x}_{j,i_j} - y_{l,i_l}, y_{j,i_j} - \OS{x}_{l,i_l},\OS{x}_{l,i_l}^{n_j}y_{l,i_l}^{n_l})
\end{EQ}
\begin{defin}\fslabel{def_plumbing}
The \emph{plumbing scheme} for $\Gamma$ is the scheme $P$ we obtain from the $W_l$ with the gluing above. By removing the
last equation we see that it is embedded into a smooth surface $X_P$.
\end{defin}
\begin{leme}\fslabel{lem_local_iso_to_Wl}
Let $Z$ be a scheme such that $\Gamma_Z$ is potentially taut, and let $P$ the plumbing scheme for $\Gamma_Z$.
Then $Z$ and $P$ are locally isomorphic.
\end{leme}
\begin{proof}
We want to show, that for every $W_l$ we find a $V_l \subset Z$ with $V_l \iso W_l$.
Because we may transfer any three points on a $\PR_k^1$ to $0$, $\infty$ and $1$, we easily get
$V_l$ such that $\red{(V_l)}$ is isomorphic to $\red{(W_l)}$.

Now we want to extend the isomorphism between $\red{(W_l)}$ and $\red{(V_l)}$ to one between $W_l$ and $V_l$.
We do this as in the previous section. That is, we thicken either the $\PR^1_k$-part or one of the affine arms
from the $n$-th to the $(n+1)$-th infinitesimal neighbourhood and show that we can extend the isomorphism.

First we observe that extending at the affine parts is always possible because we can always extend locally on each affine
arm via Proposition \ref{prop_local_thicken}, and this glues because the extensions are trivial on the $\PR^1_k$-part
because there is simply nothing to extend.
So we first extend at the $\PR^1_k$-part as much as needed, and then simply extend at the affine parts.
The only difficult step for the $\PR^1_k$-part is the first one. For this one has to calculate the
$\check{H}^1(W_l,\Aut_{\red{(W_l)}}(W_l))$ by hand. We omit this local calculation which shows that two such schemes are isomorphic
if and only if the $\nu_l$ are equal.

Now with $\nu_l \ge 1$ and $t_l \le 3$ and $B_{l_0} \iso \PR_K^1$ we can use Corollary \ref{cor_vanish_by_num} with $<0$ replaced by
$\le 1$ to show that all cohomology groups in Corollary \ref{cor_relevant_h1} vanish.
This shows that we can indeed choose $W_l$ and $V_l$ isomorphic.

Moreover, if one uses the established isomorphism between $W_l$ and $V_l$, then one knows, that $Z$ must be isomorphic to a scheme
we get by gluing the $W_l$ in an other kind then for $P$. By calculating the automorphisms of
$\OS{A}_{lj}/ (\OS{x}_{l,i_l}^{n_j}y_{l,i_l}^{n_l})$ one gets that this glueing is given by changing the relations of
\eqref{eq_charts_glueing_P}
into $\OS{x}_{j,i_j} - y_{l,i_l}(a_{y,l,j}+ \OS{x}_{l,i_l}y_{l,i_l}p_{y,l,j})$ and
$y_{j,i_j} -\OS{x}_{l,i_l}(a_{x,l,j}+ \OS{x}_{l,i_l}y_{l,i_l}p_{x,l,j})$
with
$a_{x,l,j},a_{y,l,j} \in k^{\times}$ and $p_{x,l,j},p_{y,l,j} \in \OS{A}_{lj}$
\end{proof}
For the coming calculation we need the following variant of the well-known Mayer--Vietoris sequence.
With the standard notations for Čech cohomology, we have:
\begin{prop}\fslabel{prop_mayer-vietoris}
Let $X$ be a separated space, and $\F$ a sheaf of groups on $X$ and $I$ a totally ordered set.
Further let $\U =\{(U_i)\}_{i \in I}$ be an open covering of $X$. There is an exact sequence
\[ 0 \lRa H^0(X,\F) \lRa \prod\limits_{i \in I} H^0(U_i, \F|_{U_j}) \lRa Z^1(\U, \F) \Ra H^{1}(X,\F) \lRa
\prod\limits_{i \in I} H^1(U_i, \F|_{U_j}) \]
\end{prop}

\begin{proof}
From the definition of Čech cohomology, we get an exact sequence
\begin{EQ}\fseqlabel{seq_mv_part1}
0 \lRa H^0(X,\F) \lRa \prod\limits_{i \in I} H^0(U_i, \F|_{U_i}) \lRa Z^1(\U, \F)
\stackrel{\lambda}{\lRa} \check{H}^1(\U, \F) \lRa 0
\end{EQ}
This sequence gives us the first three terms of our sequence.
Then by Proposition 5.1.1 of \cite{grothendieck1955general},
we know that the natural map
$\tau: \check{H}^1(\U, \F) \Ra H^1(X,\F)$ is injective.
Now we interpret $H^1(X,\F)= \check{H}^1(X, \F)$ as the group of $\F$-torsors.
If we have a $\F$-torsor, then by restricting to $U_i$ we get a $\F|_{U_i}$-torsor.
If we take the direct sum over all these restrictions, we get a map $H^{1}(X,\F) \Ra \prod\limits_{i \in I} H^1(U_i, \F|_{U_i})$,
and the kernel of this map are exactly the torsors trivialized by $\U$. But those are given by $\check{H}^1(\U, \F)$.
Summarizing we get that
\begin{EQ}\fseqlabel{seq_mv_part2}
0 \lRa \check{H}^1(\U, \F) \stackrel{\tau}{\lRa} H^1(X,\F) \lRa \prod\limits_{j \in I_n} H^1(U_j, \F|_{U_j})
\end{EQ}
is exact and thus also the concatenation of \eqref{seq_mv_part1} and \eqref{seq_mv_part2} via $\tau \circ \lambda$,
which is the Mayer--Vietoris sequence we wanted.
\end{proof}


Now we show that, at least if all $n_l$ are prime to $p$, we may calculate $H^1(P,\Theta_P)$
by calculating the rank of a matrix $M_P$ over $k$ (Recall that $\Theta_P=\SHom_{\MO_P}(\kd{P}, \MO_P)$).
For this we first review Laufers proof that $H^1(P, \Theta_P)$
is isomorphic to the quotient of two --- a priori infinite dimensional --- $k$-vector spaces, and the reduction of
this quotient to the quotient of two finite dimensional $k$-vector spaces. During this we look at the differences between
$p=1$ and $p>1$. Finally we construct the matrix $M_P$.

First we want to use the Mayer--Vietoris sequence to reduce the calculation of $H^1(P,\Theta_{P})$ to a quotient.
For this we set $I=\{1,.. .,n\}$ and $\U  = \{W_l\}_{l \in I}$.
Then Proposition \ref{prop_mayer-vietoris} provides us with an exact sequence
\begin{EQ}[rcl]\fseqlabel{seq_mv_for_P}
0 &\lRa& H^0(P,\Theta_{P}) \lRa \bigoplus\limits_{l =1}^n H^0(W_l, \Theta_{P}|_{W_l}) \lRa Z^1(\U, \Theta_{P})\\
&\lRa& H^{1}(P,\Theta_{P}) \lRa \bigoplus\limits_{l=1}^{n} H^1(W_l, \Theta_{P}|_{W_l})
\end{EQ}
Explicit calculations (which we again omit) show, that one has $H^1(W_l, \Theta_{P}|_{W_l})=0$ if and only if $\gcd(p,n_l)=1$.
So the last term of \eqref{seq_mv_for_P} vanishes if and only if all $\gcd(p,n_l)=1$.

For the third term of \eqref{seq_mv_for_P} we take a direct limit:
Choose a decreasing system of open coverings $\U^j=\{U^j_l\}$, $j \ge 0$ such that for every $l$ we have
$E_l \subset U_l^j \subset W_l$ and $E_l =\bigcap\limits_{j\in \N} U_l^j$.
Then by taking direct limit we get for the third term of \eqref{seq_mv_for_P}:
\begin{EQ}\fseqlabel{eq_dir_sum_stalks}
\limindo{j \in \N} Z^1(\U^j, \Theta_{P}) =
\bigoplus\limits_{\underset{x_{l_0,l_1} \in E_{l_0} \cap E_{l_1}}{(l_0,l_1)\in I^{2}}} \Theta_{P, x_{l_0,l_1}}
\end{EQ}
And for the second term we define \emph{the generalized stalk of $\Theta_{P}$ at the closed subset $E_l$} as
\begin{EQ}\fseqlabel{eq_def_general_stalk}
\Theta_{P,E_l} = \limindo{j \in \N} H^0(U_l^j, \Theta_{P}).
\end{EQ}
So if we use the Mayer--Vietoris argument \eqref{seq_mv_for_P} for every $\U^j$ and take the direct limit, we get an exact sequence
\begin{EQ}[rcl]\fseqlabel{seq_mv_for_P_after_limit}
0 &\lRa& H^0(P,\Theta_{P}) \lRa \bigoplus\limits_{l =1}^n \Theta_{P,E_l} \overset{\rho_P}{\lRa}
\bigoplus\limits_{\underset{x_{l_0,l_1} \in E_{l_0} \cap E_{l_1}}{(l_0,l_1)\in I^{2}}} \Theta_{P, x_{l_0,l_1}} 
\lRa H^{1}(P,\Theta_{P})\\
&\lRa& \bigoplus\limits_{\underset{\gcd(p,n_l)\not=1}{l \in I}} \limindo{j \in \N} H^1(U_l^j, \Theta_{P}|_{U_l^j})
\end{EQ}
in particular, we get the reduction we wanted:
\begin{leme}\fslabel{lem_cohom_quot}
If $\gcd(p,n_l)=1$ for all irreducible components $E_l$, then
\begin{EQ}\fseqlabel{cal_H^1}
H^{1}(P,\Theta_{P})  \iso \bigoplus\limits_{\underset{x_{l_0,l_1} \in E_{l_0} \cap E_{l_1}}{(l_0,l_1)\in I^{2}}}
     \Theta_{P, x_{l_0,l_1}}\bigg/\rho_P(\bigoplus\limits_{l =1}^n \Theta_{P,E_l})
\end{EQ}
\end{leme}
\begin{rem}\fslabel{rem_lem_cohom_quot_for_Z}
If we do this also for $H^{1}(Z,\Theta_{Z})$ the terms we get in \eqref{eq_dir_sum_stalks} and \eqref{eq_def_general_stalk}
are isomorphic to those of $P$. So the only term in \eqref{seq_mv_for_P_after_limit} which differs is
the map $\rho_P$ which is replaced by a map $\rho_Z$ and the difference depends on the differences in the glueing of $Z$ and $P$.

In particular one might reformulate Lemma \ref{lem_cohom_quot} for $Z$.
\end{rem}
Now we want to reduce the calculation of the quotient in \eqref{cal_H^1} to a quotient of finite dimensional vector spaces.
For this we look at the elements of $\Theta_{P, x_{l,j}}$ for $E_{l} \cap E_{lj} \not=\emptyset$ and show that all but finitely many
of them are always in the image of $\rho_P$. Every element of this $\Theta_{P, x_{l,j}}$ is of the form
\begin{EQ}\fseqlabel{elements_Theta_P,x_l,j}
     \suml_{s =\delta_j} \suml_{t=0}\alpha_{st} x_{l,i}^sy_{l,i}^t\frac{\del}{\del x_{l,i}}
     + \suml_{ u = 0} \suml_{v = \delta_l} \beta_{uv} x_{l,i}^uy_{l,i}^v\frac{\del}{\del y_{l,i}}
\end{EQ}
with $\delta_l=1$ if $\gcd(n_l,p)=1$ and $0$ else, $\alpha_{st},\beta_{uv} \in k$ and $i$ equals $0$ or $1$,
depending on the chart of $W_l$ in which we find $x_{l,j}$.
To simplify the notation, for the next two paragraphs we assume without any loss of generality  $i=0$ and $j=j_1$.

Like in the last paragraph of Page 85 of \cite{MR0367277} we get the following two lemmata reducing the elements of
\eqref{elements_Theta_P,x_l,j} which are relevant for the calculation of $H^1(P,\Theta_{P})$ to only finitely many:
\begin{leme}\fslabel{lem_fixed_contrib_Pl}
For all $a\ge n_l$, $b  \ge 0$ there are elements $f,g \in \Theta_{P,E_l}$ with
\[\rho_P(f)=y_{l,0}^ax_{l,0}^{\delta_{j_1} +b} \frac{\del}{\del x_{l,0}}\text{ and } \rho_P(g)=y_{l,0}^ax_{l,0}^{b} \frac{\del}{\del y_{l,0}}\]
in $\Theta_{P, x_{l,j}}$ and $\rho_P(f)= \rho_P(g)= 0$ at every other stalk $\Theta_{P, x_{i_1,i_2}}$.
\end{leme}
\begin{leme}\fslabel{lem_fixed_contrib_Pj}
For all $a \ge 0$, $b \ge n_{j}$ there are elements $f,g \in \Theta_{P,E_j}$ with
\[\rho_P(f)=y_{l,0}^{a}x_{l,0}^{b} \frac{\del}{\del x_{l,0}}\text{ and } \rho_P(g)=y_{l,0}^{\delta_l+a}x_{l,0}^{b} \frac{\del}{\del y_{l,0}}\]
in $\Theta_{P, x_{l,j}}$ and $\rho_P(f)= \rho_P(g)= 0$ at every other stalk $\Theta_{P, x_{i_1,i_2}}$.
\end{leme}
This shows: For the calculation of $H^1(P, \Theta_P)$, we only have to know whether for all $l$ the following finitely many elements
of $\Theta_{P, x_{l,j}}$ are in the image of $\rho_p$:
\begin{EQ}\fseqlabel{eq_rel_stalk}
     \suml_{s = \delta_{j}}^{ n_{j} -1} \suml_{t = 0}^{ n_l-1} \alpha_{st} x_{l,0}^sy_{l,0}^t\frac{\del}{\del x_{l,0}}
     + \suml_{ u = 0}^{n_{j}-1} \suml_{v = \delta_l}^{n_l-1} \beta_{uv} x_{l,0}^uy_{l,0}^v\frac{\del}{\del y_{l,0}}
\end{EQ}
Now we have a closer look at the remaining elements of $ \Theta_{P,E_l}$. These are only finitely many, but depending on the value
of $t_l$ we get different lists. For better readability we assume $\gcd(p,n_l)=1$ for all $l$.
If $\gcd(p,n_l)\not=1$ for some $l$, then the lists remain finite, but we get some extra terms.
For the calculations we use the given covering for the $W_l$.

Depending on $t_l$ the elements of the generalized stalk $\Theta_{P,E_l}$ are contained in the following lists:
In all three cases the $\frac{\del}{\del y_{l,0}}$ are with $0 <b$ and $0\le a \le \nu_l (b-1)$ given by:
\begin{EQ}\fseqlabel{eq_items_PEl_y}
     x_{l,0}^ay_{l,0}^b \frac{\del}{\del y_{l,0}} = x_{l,1}^{\nu_l(b-1)-a} y_{l,1}^b\frac{\del}{\del y_{l,1}} 
\end{EQ}
For $\frac{\del}{\del x_{l,0}}$ we have look at $t_l$.
For $t_l=1,2$ we have with $0 \le b$ and $ 0< a \le (\nu_l b +1)$:
\begin{EQ}\fseqlabel{eq_items_PEl_x_tl12}
     x_{l,0}^ay_{l,0}^b \frac{\del}{\del x_{l,0}} = -x_{l,1}^{\nu_l b-a+2} y_{l,1}^b\frac{\del}{\del x_{l,1}}
     + \nu_l x_{l,1}^{\nu_l b-a+1} y_{l,1}^{b+1}\frac{\del}{\del y_{l,1}}
\end{EQ}
For $t_l=1$ we have additionally for $0 \le b$:
\begin{EQ}\fseqlabel{eq_items_PEl_x_tl1}
y_{l,1}^b \frac{\del}{\del x_{l,1}} = -x_{l,0}^{\nu_l b+2} y_{l,0}^b\frac{\del}{\del x_{l,0}}
     + \nu_l x_{l,0}^{\nu_l b+1} y_{l,0}^{b+1}\frac{\del}{\del y_{l,0}}
\end{EQ}
Finally, for $t_l=3$ we have for $0 < b$ and $0 < a \le \nu_l b$:
\begin{EQ}\fseqlabel{eq_items_PEl_x_tl3}
     x_{l,0}^ay_{l,0}^b(x_{l,0}-1) \frac{\del}{\del x_{l,0}} = x_{l,1}^{\nu_l b-a+1} y_{l,1}^b(x_{l,1}-1)\frac{\del}{\del x_{l,1}}
     - \nu_l x_{l,1}^{\nu_l b-a} y_{l,1}^{b+1}(x_{l,1}-1)\frac{\del}{\del y_{l,1}}
\end{EQ}
From this and Lemma \ref{lem_cohom_quot} we immediately get the following proposition:
\begin{prop}\fslabel{prop_h1_plumbing}
If $\gcd(p,n_l)=1$ for all $l$, then $H^1(P,\Theta_P) =0$ if and only if the image of
\eqref{eq_items_PEl_y}, \eqref{eq_items_PEl_x_tl1}, \eqref{eq_items_PEl_x_tl12} or \eqref{eq_items_PEl_x_tl3} under $\rho_P$
 generates all elements of the form \eqref{eq_rel_stalk}.
\end{prop}
A nice consequence of this proposition is that it provides a way to actually calculate $h^1(P,\Theta_P)$.
For this we construct a $r_P \times c_P$ matrix $M_P$ over $k$ in the following way:
For every point $x_{l,j}$ and every element of \eqref{eq_rel_stalk} we add one row to $M_P$. Then for every $P_l$ and every Element of
\eqref{eq_items_PEl_y}, \eqref{eq_items_PEl_x_tl1}, \eqref{eq_items_PEl_x_tl12} or \eqref{eq_items_PEl_x_tl3} we add a column to $M_P$.
The entries in $M_P$ are simply the coefficients of the element associated to the column as an expansion in the element associated
to the row. Note that, by construction,
the entries of $M_P$ are integers.
Also by the construction of $M_P$ we get the following corollary of Proposition \ref{prop_h1_plumbing}:
\begin{cor}\fslabel{cor_h1_plumbing_rank}
If $\gcd(p,n_l)=1$ for all $l$, then $h^1(P,\Theta_P) = r_P - \rank(M_P)$
\end{cor}
\begin{rem}
Proposition \ref{prop_h1_plumbing} and Corollary \ref{cor_h1_plumbing_rank} work analogously for $H^1(Z,\Theta_Z)$,
but $M_Z$ is in practice much harder to write down explicitly than $M_P$.
\end{rem}
As a consequence of the corollary we get the following comparison between $p=1$ and $p>1$:
\begin{prop}\fslabel{prop_dim_h1_goes_up}
Let $P_1$ be a plumbing scheme over $\C$, and for all $p>1$ with $\gcd(p,n_l)=1$ for all $l$ let $P_p$ be the plumbing scheme for the
same dual graph over an algebraically closed field of characteristic $p$. Then we have
\[h^1(P_1,\Theta_{P_1}) \le h^1(P_p,\Theta_{P_p})\]
and equality for all but finitely many $p$.
\end{prop}
\begin{proof}
By Corollary \ref{cor_h1_plumbing_rank} we have
$h^1(P_p,\Theta_{P_p}) = r_{P_p} - \rank(M_{P_p})$. By construction we get $M_{P_p}$ for $p>1$ if we take all entries of $M_{P_1}$
modulo $p$. In particular $r_p$ is independent of $p$. Now $\rank(M_{P_1})=m$ is equivalent to the existence of one non-vanishing
$m \times m$ minor, and all $(m+1)\times (m+1)$ minors vanish. But the minors of $M_{P_p}$ are just the minors of $M_{P_1}$ modulo $p$,
so the rank can only decrease, thus the $h^1(P_p,\Theta_{P_p})$ can only increase.

Finally the rank decreases if and only if $p>1$ divides all $m \times m$ minors of $M_{P_1}$, so it
decreases for exactly the prime factors of the gcd of all non vanishing $m \times m$ minors of $M_{P_1}$.
\end{proof}

Our goal is to show, that $H^1(P,\Theta_P) =0$ implies that every $Z$ combinatorially equivalent to $P$ is already isomorphic
to $P$.
We prove this later, but now we are able to prove that $H^1(P,\Theta_P) =0$ already implies
$H^1(Z,\Theta_Z) =0$ for all $Z$ combinatorial equivalent to $P$, which of course is a necessary condition for
$Z$ to be isomorphic to $P$:
\begin{prop}\fslabel{prop_van_plumbing_impl_van}
If $\gcd(p,n_l)=1$ for all $l$, and we have $H^1(P,\Theta_P) =0$, then we have $H^1(Z,\Theta_Z) =0$ for all $Z$
combinatorial equivalent to $P$.
\end{prop}
\begin{proof}
By Lemma \ref{lem_cohom_quot} we have to prove that the surjectivity of $\rho_P$ on every $\Theta_{P,x_{l,j}}$ implies the
surjectivity of $\rho_Z$ on every $\Theta_{Z,x_{l,j}}$.
By Remark \ref{rem_lem_cohom_quot_for_Z} we know that the only difference between  $\rho_P$ and  $\rho_Z$ is the gluing.
To make this precise: We know that $\Theta_{P,x_{l,j}} \iso \Theta_{Z,x_{l,j}}$, and they are as
$k[\OS{x}_{l,i_l}, y_{l,i_l}]/( \OS{x}_{l,i_l}^{n_j}y_{l,i_l}^{n_l})$-module
generated by $\OS{x}_{l,i_l}\frac{\del}{\del \OS{x}_{l,i_l}}$ and $y_{l,i_l}\frac{\del}{\del y_{l,i_l}}$.
Now, for all $f=\OS{x}_{l,i_l}^ay_{l,i_l}^b \frac{\del}{\del \OS{x}_{l,i_l}} \in \Theta_{P,E_l}$
we have $\rho_P(f)=\rho_Z(f)=\OS{x}_{l,i_l}^ay_{l,i_l}^b \frac{\del}{\del \OS{x}_{l,i_l}}$, and the same with
$\frac{\del}{\del y_{l,i_l}}$. In particular, Lemma \ref{lem_fixed_contrib_Pl} stays true with $\rho_Z$ instead of $\rho_P$.

Next we want to look at the image of $\Theta_{P,E_j}$ in $\Theta_{P,x_{l,j}}$. Suppose we have a $f \in \Theta_{P,x_{l,j}}$
with $\rho_P(f)=y_{l,i_l}^{a}\OS{x}_{l,i_l}^{b} \frac{\del}{\del \OS{x}_{l,i_l}}$ which we also may write as
$\OS{x}_{j,i_j}^{a}y_{j,i_j}^{b} \frac{\del}{\del y_{j,i_j}}$. But then by a calculation using the charts
of given in the proof of Lemma \ref{lem_local_iso_to_Wl} we get:
\begin{EQ}\fseqlabel{eq_roh_Z_f}
\rho_Z(f)=\OS{x}_{j,i_j}^{a}y_{j,i_j}^{b} \frac{\del}{\del y_{j,i_j}}
=a_{y,l,j}^{a+1} y_{l,i_l}^a a_{x,l,j}^b\OS{x}_{l,i_l}^b \frac{\del}{\del \OS{x}_{l,i_l}} + y_{l,i_l}^{a+1} \OS{x}_{l,i_l}^{b}R_f
\end{EQ}
with some $R_f$.
Analogously, if we have some $g \in \Theta_{P,x_{l,j}}$ with
\[\rho_P(g)=y_{l,i_l}^{a}\OS{x}_{l,i_l}^{b} \frac{\del}{\del y_{l,i_l}}=x_{j,i_j}^{a}y_{j,i_j}^{b} \frac{\del}{\del \OS{x}_{j,i_j}}\]
then we have
\begin{EQ}[rcl]\fseqlabel{eq_roh_Z_g}
\rho_Z(f)&=&  \OS{x}_{j,i_j}^{a}y_{j,i_j}^{b} \frac{\del}{\del \OS{x}_{j,i_j}}\\
&=& a_{y,l,j}^{a} y_{l,i_l}^a a_{x,l,j}^b\OS{x}_{l,i_l}^b (a_{x,l,j}^2\OS{x}_{l,i_l}^2 p_{y,j,l}\frac{\del}{\del \OS{x}_{l,i_l}}+
a_{y,l,j}\frac{\del}{\del y_{l,i_l}}) + y_{l,i_l}^{a+1} \OS{x}_{l,i_l}^{b}R_g
\end{EQ}
Now we want to prove that we have Lemma \ref{lem_fixed_contrib_Pj} for $Z$. For this, let $b \ge n_j$.
Because we have Lemma \ref{lem_fixed_contrib_Pl} for $Z$, we only have to care for $a < n_l$. For $a = n_l-1$ the terms
$y_{l,i_l}^{a+1} \OS{x}_{l,i_l}^{b}R_f$ and $y_{l,i_l}^{a+1} \OS{x}_{l,i_l}^{b}R_g$ vanish. But $a_{y,l,j}$ and $a_{x,l,j}$ 
are units in $k$, so \eqref{eq_roh_Z_f} shows us that $ y_{l,i_l}^a\OS{x}_{l,i_l}^b \frac{\del}{\del \OS{x}_{l,i_l}}$ is in the
image of $\rho_Z$, and with this \eqref{eq_roh_Z_g} shows that also
$ y_{l,i_l}^a\OS{x}_{l,i_l}^b \frac{\del}{\del y_{l,i_l}}$ is in the image of $\rho_Z$. So by doing inverse induction on $a$ we see
that we have Lemma \ref{lem_fixed_contrib_Pj} for $Z$.

It remains to show that the surjectivity of $\rho_P$ implies, that for $a < n_l$ and $b< n_l$ also
$y_{l,i_l}^a\OS{x}_{l,i_l}^b\frac{\del}{\del \OS{x}_{l,i_l}}$ and $y_{l,i_l}^a\OS{x}_{l,i_l}^b\frac{\del}{\del y_{l,i_l}}$
are in the image of $\rho_Z$. But with \eqref{eq_roh_Z_f} and \eqref{eq_roh_Z_g} this follows analogously to the argumentation before.
We only have to to a double inverse induction on $a+b$: We start with $a = n_l-1$ and $b=n_j-1$. In each step we reduce $a$ until
$a=0$ and then we reduce $b$ by one and start again with $a = n_l-1$.
\end{proof}
\begin{rem}
The inverse of this theorem does not hold. There is a counterexample with $H^1(P,\Theta_P) =\C$ but $H^1(Z,\Theta_Z) =0$ of Laufer
(\cite{MR0367277}, §4 end of page 93).
\end{rem}
To show that $H^1(P,\Theta_P) =0$ implies $\CEQs{P}=\{[(P,X_P)]\}$ we need the deformation theory of $P$.
Recall that for a $k$-scheme $X$, a \emph{deformation $\eta$ of $X$ over $(S,s)$} is a cartesian diagram
\[\xymatrix{
\ar@{}[d]|{\eta:} &X \ar[r] \ar[d] & \X \ar[d]^{\pi}\\
& \spec(k) \ar[r]^s & S
}\]
where $\pi$ is flat and surjective, $S$ is connected and $s$ is a $k$-rational point of $S$.
We say that $\eta$ is \emph{locally trivial} if for every point $x \in X$ we find an open
neighbourhood $U_x \subset X$ such that the induced deformation of $U_x$
\[\xymatrix{
\ar@{}[d]|{\eta|_{U_x}:} &U_x \ar[r] \ar[d] & \X|_{U_x} \ar[d]^{\pi}\\
& \spec(k) \ar[r]^s & S
}\]
is isomorphic to the trivial deformation of $U_x$. Usually locally trivial deformations are of little interest. For example,
if $X$ and $S$ is are smooth curves, $\eta$ being locally trivial implies that every smooth closed fiber of $\pi$ is already isomorphic
to $X$. On the other hand, the functor of locally trivial deformations is better understandable,
in particular if we assume the schemes $S$ to be spectra of artinian rings. 

So for a $k$-scheme $X$ we define the following functor from the category of connected schemes together with a
$k$-rational point to sets:
\[\cdef'_X(S,s)=\{\text{locally trivial deformations of $X$ over $(S,s)$}\}/\text{isomorphism}\]
Now, for a cycle supported on the exceptional locus we always find a locally trivial deformation into the plumbing scheme
for its dual graph: 
\begin{prop}\fslabel{prop_defo_exists}
Let $P$ be the plumbing scheme for a potentially taut dual graph, and take $[(Z,X)] \in \CEQs{P}$.
Then there exists an integral affine scheme
$Y$, a $k$-rational point $ y' \in Y$ and $\eta \in \cdef'_{Z}(Y,y)$ with $\pi^{-1}(y') \iso P$.
\end{prop}
\begin{proof}
From the proof of Lemma \ref{lem_local_iso_to_Wl} we know that for $Z$ the glueing along every $W_{lj} \not=\emptyset$ is done via
$x_{j,i_j} = y_{l,i_l}(a_{y,l,j}+ x_{l,i_l}y_{l,i_l}p_{y,l,j})$ and $y_{j,i_j} = x_{l,i_l}(a_{x,l,j}+ x_{l,i_l}y_{l,i_l}p_{x,l,j})$.

Let $A = k[u_{x,l,j}, u_{y,l,j},u_{x,l,j}^{-1}, u_{y,l,j}^{-1},t_x, t_y]$ (with $lj$ running over all $lj$ such that
$W_{lj} \not= \emptyset$), and $Y=\spec(A)$.
We define $\X$ as follows: We glue the $W_l \times Y$ along the $W_{lj}\times Y$ via
$x_{j,i_j} = y_{l,i_l}(u_{y,l,j}+ x_{l,i_l}y_{l,i_l}p_{y,l,j} t_y)$ and
$y_{j,i_j} = x_{l,i_l}(u_{x,l,j}+ x_{l,i_l}y_{l,i_l}p_{x,l,j}t_x)$
which defines an automorphism, because the right factors are of the form ``invertible + nilpotent''.

Let now $\pi$ be the projection. By construction of $\X$ we have
$P \iso \pi^{-1}(1,1,\dots,1,1,0,0)$ and
$Z \iso \pi^{-1}(a_{y,1,2},a_{x,1,2},\dots,a_{x,l,n},a_{y,l,n},1,1)$.

Now $\pi$ is locally trivial by construction, in particular flat.
\end{proof}
Now we are able to prove that  $H^1(P,\Theta_P) =\{0\}$ implies $\CEQs{P}= \{[P]\}$.
\begin{prop}\fslabel{prop_CEQsP_triv_if_h1_0}
Let $P$ be the plumbing scheme for a potentially taut dual graph with $\gcd(p,n_l)=1$ for all $l$.
If $H^1(P,\Theta_P) =0$, then $\CEQs{P}= \{[(P,X_P)]\}$.
\end{prop}
\begin{proof}
Let $Z$ be any scheme combinatorially equivalent to $P$.
From Proposition \ref{prop_defo_exists} we get a locally trivial deformation $\eta$
of $Z$ into $P$. Now the base of this deformation is an integral affine scheme, so, via localisation, we may assume that we have
$Y=\spec(R)$, where $R$ is an integral semi-local ring with exactly two maximal ideals $m_1$ and $m_2$.
Let $y_i$ be the point given by $m_i$, and let $X_i=\pi^{-1}(y_i)$. Suppose that we have $X_1 \iso Z$ and $X_2 \iso P$.

Localizing further we get two local rings $(R_1,m_1)$ and $(R_2,m_2)$ both with residue field $k$ and a common quotient field $K$.
Then we have $X_i \iso \X \times_{\spec(R)}\spec(R_i/m_i)$.
Let $\OH{R}_i$ be the completion of $R_i$, and $K_i$ the quotient field of $\OH{R}_i$.
By the universal property of the quotient field we get maps $K \Ra K_i$, and there exists a field $\OS{K}$ containing
$K_1$ and $K_2$.

From this we get for the spectra, using standard properties of the fibre-product:
\begin{EQ}
 \X \times_{\spec(R_1)} \spec(\OH{R}_1)\times_{\spec(\OH{R}_1)} \spec(\OS{K})
\iso  \X \times_{\spec(R_2)} \spec(\OH{R}_2)\times_{\spec(\OH{R}_2)} \spec(\OS{K})
\end{EQ}
From this we get
\[ Z  \times_{\spec(k)} \spec(\OS{K}) \iso P \times_{\spec(k)} \spec(\OS{K}), \]
if we show
\[\X \times_{\spec(R_i)} \spec(\OH{R}_i) \iso X_i  \times_{\spec(k)} \spec(\OH{R}_i),\]
For this we look at the functors $\Def'_P$ and $\Def'_Z$ and restrict them to spectra of local artinian $k$-algebras.
By Theorem 2.4.1 of \cite{MR2247603} the tangent-space of this functors are $H^1(P,\Theta_P)$ and $H^1(Z,\Theta_Z)$ respectively
and thus trivial; the first one by the assumption, the second one by Proposition \ref{prop_van_plumbing_impl_van}.
So by the same Theorem of \cite{MR2247603}, they have a semi-universal element.
Now Proposition 2.2.8 of \cite{MR2247603} tells us $\Def'_P=\Hom(k, \;\;\;) = \Def'_Z$.\
From this we get $\Def'_Z(\OH{R}_1)=\Hom(k,\OH{R}_1)$ and $\Def'_P(\OH{R}_2)=\Hom(k,\OH{R}_2)$,
or in other words, for every $n$ we have:
\[\X \times_{\spec(R_i)} \spec(\OH{R}_i/m_i^{n+1}) \iso  X_i \times_{\spec(k)} \spec(\OH{R}_i/m_i^{n+1})\]
That is, as formal schemes we have $\OH{\X}|_{X_i} \iso  \OH{X_i}|_{X_i}$,
which by \cite{MR0217085}, 5.4.1 gives us
\[\X \times_{\spec(R_i)} \spec(\OH{R}_i) \iso X_i  \times_{\spec(k)} \spec(\OH{R}_i),\]
as wanted.

So we have not yet that $P$ and $Z$ are isomorphic, but we know that they are isomorphic after base change to some field extension
of $k$. Now we want to get the isomorphism between $P$ and $Z$ from this isomorphism. For this 
we take a look at the isomorphism functor mapping a schemes $S$ to $\Iso_k(Z \times_k S, P \times_k S)$
Fortunately, because $Z$ and $P$ are proper, one-dimensional schemes over a field and thus projective,
by \cite{MR1611822} this functor is represented by a scheme $I$ locally of finite type over $k$.
So we know $I(\OS{K}) \not= \emptyset$, thus $I$ is not the empty scheme and thus has a $\kk{k}=k$-rational point.
But this point corresponds to an isomorphism between $Z$ and $P$, which finishes the proof.
\end{proof}
For singularities this has the following consequence:
\begin{prop}\fslabel{prop_h1_0_implies_taut}
Let $S$ be a normal two-dimensional singularity, $\OS{Z}=\suml_{l=1}^n n_l E_l$ an anti-ample divisor for $S$
with $\gcd(p,n_l)=1$ for all $l$.
Further let $\nu$ be the significant multiplicity for $\OS{Z}$.
If $P$ is the plumbing scheme for $\Gamma_{\nu\OS{Z}}$, then $S$ is taut if $H^1(P,\Theta_P)=0$.
\end{prop}
\begin{proof}
By Proposition \ref{prop_CEQsP_triv_if_h1_0} we have $\CEQs{P} = \CEQs{\nu \OS{Z}} = \{[(\nu\OS{Z},X)]\}$, so $S$ is taut
by Corollary \ref{cor_taut_if_neqs_trivial}.
\end{proof}
Finally we are able to prove the next comparison between $p=1$ and $p>1$:
\begin{prop}\fslabel{prop_0_comb_def_implies_most_p_comb_def}
Let $\Gamma$ be the dual graph of some plumbing scheme $P_1$ over $\C$, and for all $p>1$ with $\gcd(p,n_l)=1$ for all $l$ let $P_p$
be the plumbing scheme for $\Gamma$ over an algebraically closed field of characteristic $p$. Then $\CEQs{P_1}=\{[(P_1,X_{P_1})]\}$
implies $\CEQs{P_p}=\{[(P_p,X_{P_p})]\}$ for all but finitely many $p$.
\end{prop}
\begin{proof}
By \cite{MR0367277}, Theorem 3.9 from $\CEQs{P_1}=\{[(P_1,X_{P_1})]\}$ we get $h^1(P_1,\Theta_{P_1}) =0$,
which by Proposition \ref{prop_dim_h1_goes_up} implies $h^1(P_p,\Theta_{P_p})=0$ for all but finitely many $p>1$.
So we get $\CEQs{P_p}=\{[(P_p,X_{P_p})]\}$ for the good $p$ with Proposition \ref{prop_CEQsP_triv_if_h1_0}.
\end{proof}
If we transfer this to the tautness of normal two-dimensional singularities, we get our main theorem:
\begin{thm}\fslabel{thm_0_taut_implies_most_p_taut}
Let $S_1$ be a normal two-dimensional singularity over $\C$ with dual graph $\Gamma$. For all primes $p$ let
$S_p$ be a singularity over an algebraically closed field of characteristic $p$ with dual graph $\Gamma$.
If $S_1$ is taut, then $S_p$ is taut for all but finitely many $p$.
\end{thm}
\begin{proof}
First we note, that with Lemma \ref{lem_contr_scheme} one gets the existence of at least one $S_p$ for every $p>1$.
Further, by Lemma \ref{lem_necessary_cond_for_taut} we may assume that $\Gamma$ is potentially taut.

Now let $\OS{Z}_p$ be an anti-ample divisor for $S_p$ with significant multiplicity $\nu_p$.
By Corollary \ref{cor_taut_if_neqs_trivial} the tautness of $S_1$ implies $\CEQs{\nu_1 \OS{Z}_1} = \{[(\nu_1 \OS{Z}_1,X)]\}$.

The coefficients of $\OS{Z}_p$ are defined by combinatorial data independed of the ground field, so we can assume that all coefficients
of $\OS{Z}_1$ and $\OS{Z}_p$ are equal. By the construction of $\nu$ we have $\nu_p=\nu_1$ and 
 $\gcd(p,n_l)=1$ for all $l$ for all but finitely many $p$.

Let now $P_p$ be the plumbing scheme for $\nu_p \OS{Z}_p$.
We have $\CEQs{P_1} = \{[(\nu_1 \OS{Z}_1,X)]\}$, so we are in the situation of
Proposition \ref{prop_0_comb_def_implies_most_p_comb_def}, that is we have $\CEQs{P_p}= \{[(P_p,X_{P_p})]\}$ for all but finitely
many $p$. So for all those $p$ we get the tautness of $S_p$ from Corollary \ref{cor_taut_if_neqs_trivial}.
\end{proof}

\section{Open questions}

So far we have mainly shown which of Laufers results work also for $p>1$. Now we want to discuss the result we are not able to
carry over. In particular, we want to give evidences for a conjectural picture for those $p$ with a strict inequality in Proposition
\ref{prop_dim_h1_goes_up}.

First note that for a given $\Gamma$ we can compute the good and the bad $p$ for this $\Gamma$. With ``good'' we mean that
for this $p$ the tautness of $S_1$ implies the tautness of $S_p$.
The two places in the proof of Theorem \ref{thm_0_taut_implies_most_p_taut} where we had to exclude some primes can be healed.
The first place is very simple:
For all $p$ with $\nu_p = \nu_1+1$ we simply do the proof again, with $\nu_1$ replaced by $\nu_1+1$.
The second place needs a little more thinking, but with Lemma~\ref{lem_prim_anti_ample_cycle} we see
that we can always choose the coefficients of $\OS{Z}_1$ prim to every fixed $p$.
So going through the proof finitely many times shows that a $p$ is good if it is not one of the finitely many primes excluded by
Proposition \ref{prop_0_comb_def_implies_most_p_comb_def}. That is $p$ is good if and only if we have equality in
Proposition \ref{prop_dim_h1_goes_up}. So theoretically we are able to calculate all good $p$ for a given singularity, but in practice
the matrix $M_{P_1}$ is huge.

If we know the bad prime $p$ for $\Gamma$, and if $S_p$ is a $\Gamma$-singularity in characteristic $p$, we conjecture
that $S_p$ is not taut. Over $\C$ Laufers Theorem 3.9 of \cite{MR0367277}
is a stronger version of our Proposition \ref{prop_CEQsP_triv_if_h1_0}, which also has the inverse implication.
That is, it says $\CEQs{Z}=\{[(Z,X)]\}$ if and only if $H^1(P,\Theta_P) =0$.
A simple example shows, that this can not be true for $p>1$:
Take $E=E_1=\PR^1_k$ and $Z=pE_1$ and $\nu_1 > 1$.
With an explicit calculation one gets $\CEQs{2E_1}= \{[(2E_1,X)]\}$, and
in Corollary \ref{cor_induction_neqs} we have $\OS{Z}=E$ and $\nu=2$, so this implies $\CEQs{jE_1}= \{[(jE_1,X)]\}$ for all $j$.
In particular, we have $P\iso Z$.
But again a calculation in local coordinates, shows that one has
$h^1(Z, \Theta_Z)= h^1(P, \Theta_P)= \nu_1 -1$ for $p | i$.

So if we demand the $n_l$ to be prime to $p>1$,
then with Proposition \ref{prop_CEQsP_triv_if_h1_0} we have the ``if'' statement of Laufer's Theorem 3.9, and we think that this is also
the modification needed for the ``only if'' direction, so we propose the following conjecture:
\begin{conj}\fslabel{conj_P_Z_iso_iff_h1_0}
Let $P$ be the plumbing scheme for a potentially taut dual graph with $\gcd(p,n_l)=1$ for all $l$.
Then we have $\CEQs{P}=\{[(P,X_P)]\}$ if and only if $H^1(P,\Theta_P) =0$.
\end{conj}
To find evidence for this conjecture we look at the ADE-singularities. Artin calculated a full list of
all isomorphism classes of those in all characteristics in \cite{MR0199191}.
So we look at the non-taut ADEs, and calculate $h^1(P,\Theta_P)$ for $\nu\OS{Z}$ as in Corollary
\ref{cor_taut_if_neqs_trivial}.

This calculation can be done with the help of the computer algebra system Sage on a computer with enough memory.
We will now indicate how we have done this.

First to simplify the construction of $M_P$ we want to stick to some cycle of the form $jE$, where $E=\suml_{l=1}^n E_l$.
This is no problem because if we choose $j$ bigger then $\nu \cdot \max\{n_l\}$, then with Corollary \ref{cor_osz_triv_ind_greater_triv}
we know that we have $\CEQs{jE}\iso\CEQs{\nu \OS{Z}}$. To make sure that $p$ does not divide $j$, we take the next prime
bigger than $\nu \cdot \max\{n_l\}$ as $j$ (and $j >7$, the biggest $p$ we are interested in).

We want to discuss the calculation of the significant multiplicity $\nu$ first.
By definition $\nu$ depends on $\lambda$ and $\tau_{min}$ defined  previous to Proposition \ref{prop_red_to_finite}.
The calculation of $\lambda$ depending on the $\Gamma$ is easy, in particular we have $\lambda=0$ for all ADEs,
because $p_a(E_l)=0$ and $E_l^2=-2$ for all $l$.

The calculation of $\tau_{min}$ is not so easy. Going over all possibilities needs to much time, so
we had to find a way to compute a good upper bound for $\tau_{min}$ for all ADEs.
We take $\beta_1=1$ and then we construct $\beta_{i}$ inductively as follows:
Let $\OS{Z}_{i-1}=\suml_{l=1}^n s_{l,i-1} E_l$ and $\OS{\beta}$ be the smallest integer between $1$ and $n$
such that $s_{\OS{\beta},i-1} < n_{\OS{\beta}}$ and $E_{\OS{\beta}} \cdot (\OS{Z}_{i-1} + E_{\OS{\beta}})$ is maximal among these
$\OS{\beta}$. Then we set $\beta_i=\OS{\beta}$.
If we now calculate $\tau$ for this $\beta_i$ and our $\OS{Z}$ chosen (see below) with the help of a computer, we get always $\tau=1$.
So because all $n_l$ are greater then $1$, we simply take $\nu=2$.

Now the anti-ample cycles we used are (for reasons of readability we omit the $-2$ in the dual graphs):
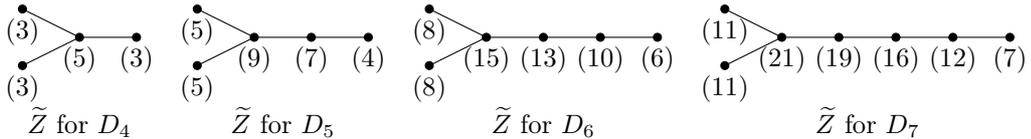
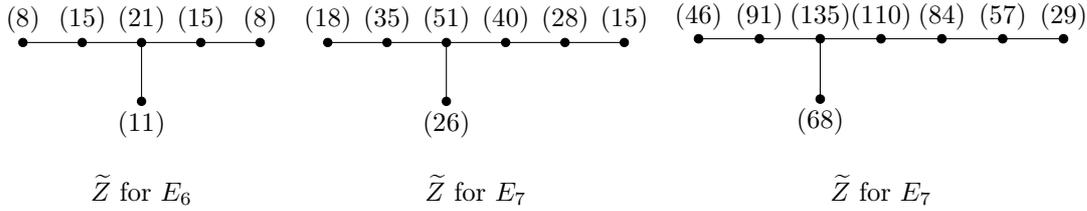
\begin{center}
\begin{tikzpicture}[scale=.75]
\filldraw [black]
     (0,0.5) circle (2pt) node[below] {$(3)$}
     (0,-0.5) circle (2pt) node[below] {$(3)$} 
     (1,0) circle (2pt) node[below] {$(5)$}
     (2,0) circle (2pt) node[below] {$(3)$};
\draw (0,0.5) -- (1,0);
\draw (0,-0.5) -- (1,0);
\draw (1,0) -- (2,0);
\draw (1,-1.5) node {$\OS{Z}$ for $D_4$};
\end{tikzpicture} 
\begin{tikzpicture}[scale=.75]
\filldraw [black]
     (0,0.5) circle (2pt) node[below] {$(5)$}
     (0,-0.5) circle (2pt) node[below] {$(5)$} 
     (1,0) circle (2pt) node[below] {$(9)$}
     (2,0) circle (2pt) node[below] {$(7)$}
     (3,0) circle (2pt) node[below] {$(4)$};
\draw (0,0.5) -- (1,0);
\draw (0,-0.5) -- (1,0);
\draw (1,0) -- (3,0);
\draw (1.5,-1.5) node {$\OS{Z}$ for $D_5$};
\end{tikzpicture}
\begin{tikzpicture}[scale=.75]
\filldraw [black]
     (0,0.5) circle (2pt) node[below] {$(8)$}
     (0,-0.5) circle (2pt) node[below] {$(8)$} 
     (1,0) circle (2pt) node[below] {$(15)$}
     (2,0) circle (2pt) node[below] {$(13)$}
     (3,0) circle (2pt) node[below] {$(10)$}
     (4,0) circle (2pt) node[below] {$(6)$};
\draw (0,0.5) -- (1,0);
\draw (0,-0.5) -- (1,0);
\draw (1,0) -- (4,0);
\draw (2,-1.5) node {$\OS{Z}$ for $D_6$};
\end{tikzpicture}
\begin{tikzpicture}[scale=.75]
\filldraw [black]
     (0,0.5) circle (2pt) node[below] {$(11)$}
     (0,-0.5) circle (2pt) node[below] {$(11)$} 
     (1,0) circle (2pt) node[below] {$(21)$}
     (2,0) circle (2pt) node[below] {$(19)$}
     (3,0) circle (2pt) node[below] {$(16)$}
     (4,0) circle (2pt) node[below] {$(12)$}
     (5,0) circle (2pt) node[below] {$(7)$};
\draw (0,0.5) -- (1,0);
\draw (0,-0.5) -- (1,0);
\draw (1,0) -- (5,0);
\draw (2.5,-1.5) node {$\OS{Z}$ for $D_7$};
\end{tikzpicture}\\
\begin{tikzpicture}[scale=.78]
\filldraw [black]
     (0,1) circle (2pt) node[above] {$(8)$}
     (1,1) circle (2pt) node[above] {$(15)$} 
     (2,1) circle (2pt) node[above] {$(21)$}
     (2,0) circle (2pt) node[below] {$(11)$}
     (3,1) circle (2pt) node[above] {$(15)$}
     (4,1) circle (2pt) node[above] {$(8)$};
\draw (0,1) -- (4,1);
\draw (2,1) -- (2,0);
\draw (2,-1.5) node {$\OS{Z}$ for $E_6$};
\end{tikzpicture}
\begin{tikzpicture}[scale=.78]
\filldraw [black]
     (0,1) circle (2pt) node[above] {$(18)$}
     (1,1) circle (2pt) node[above] {$(35)$} 
     (2,1) circle (2pt) node[above] {$(51)$}
     (2,0) circle (2pt) node[below] {$(26)$}
     (3,1) circle (2pt) node[above] {$(40)$}
     (4,1) circle (2pt) node[above] {$(28)$}
     (5,1) circle (2pt) node[above] {$(15)$};
\draw (0,1) -- (5,1);
\draw (2,1) -- (2,0);
\draw (2.5,-1.5) node {$\OS{Z}$ for $E_7$};
\end{tikzpicture}
\begin{tikzpicture}[scale=.8]
\filldraw [black]
     (0,1) circle (2pt) node[above] {$(46)$}
     (1,1) circle (2pt) node[above] {$(91)$} 
     (2,1) circle (2pt) node[above] {$(135)$}
     (2,0) circle (2pt) node[below] {$(68)$}
     (3,1) circle (2pt) node[above] {$(110)$}
     (4,1) circle (2pt) node[above] {$(84)$}
     (5,1) circle (2pt) node[above] {$(57)$}
     (6,1) circle (2pt) node[above] {$(29)$};
\draw (0,1) -- (6,1);
\draw (2,1) -- (2,0);
\draw (3,-1.5) node {$\OS{Z}$ for $E_7$};
\end{tikzpicture}
\end{center}
With some simple generators written in C++ we generated text files containing the entries of $M_P$ processable by Sage.
We chose Sage, because Sage implements an algorithm for exactly our problem (\cite{dumas2002computing}).

The main problem for the calculation is the growth of the matrix. If $pt$ is the number of intersection points $x_{lj}$ then we have
$r_P= 2\cdot pt \cdot ( j^2 -j)$, and even if this just grows quadratically, for $E_8$ and $j=203$ we have already $r_P=1024380$.
On the other hand, the matrix $M_P$ is a sparse matrix with only less then $\frac{1}{1000}$ of its entries non-zero.
It is crucial to use this fact, because without it already the text files containing the entries are several gigabyte big.
Now the result of the computations is:
\begin{center}
\begin{tabular}{|c|r|r|c|r|r|r|r|c|c|c|c|}\hline
$\Gamma$ & $\max$ & $j$ & $r_P \times c_P$       &\multicolumn{4}{c|}{rank $M_P$} &  \multicolumn{4}{c|}{ $h^1(P,\Theta_P)$}\\
       & $\{n_l\}$ &    &                        &     $p=2$ &       $3$ &       $5$ &       $7$ & $2$ & $3$ & $5$ & $7$\\\hline\hline
$D_4$ &   $5$ &  $11$ & $    660 \times     735$ & $    659$ & $    660$ & $    660$ & $    660$ & $1$ & $0$ & $0$ & $0$\\\hline
$D_5$ &   $9$ &  $19$ & $   2736 \times    2944$ & $   2735$ & $   2736$ & $   2736$ & $   2736$ & $1$ & $0$ & $0$ & $0$\\\hline
$D_6$ &  $15$ &  $31$ & $   9300 \times    9827$ & $   9298$ & $   9300$ & $   9300$ & $   9300$ & $2$ & $0$ & $0$ & $0$\\\hline
$D_7$ &  $21$ &  $43$ & $  21672 \times   22662$ & $  21670$ & $  21672$ & $  21672$ & $  21672$ & $2$ & $0$ & $0$ & $0$\\\hline
$E_6$ &  $21$ &  $43$ & $  18060 \times   19049$ & $  18059$ & $  18059$ & $  18060$ & $  18060$ & $1$ & $1$ & $0$ & $0$\\\hline
$E_7$ &  $51$ & $103$ & $ 126072 \times  131532$ & $ 126069$ & $ 126071$ & $ 126072$ & $ 126072$ & $3$ & $1$ & $0$ & $0$\\\hline
$E_8$ & $135$ & $271$ & $1024380 \times 1116997$ & $1024376$ & $1024378$ & $1024379$ & $1024380$ & $4$ & $2$ & $1$ & $0$\\\hline
\end{tabular}
\end{center}
If one compares this table with Artin's list one notices, that for all non-taut ADEs
$h^1(P,\Theta_P)+1$ is exactly the number of isomorphism classes of singularities. This suggests that Theorem 3.1 of
\cite{MR0333238} may be still true for $p>1$ if we restrict the $n_l$ as before.
So we propose a stronger version of Conjecture \ref{conj_P_Z_iso_iff_h1_0}:
\begin{conj}\fslabel{conj_no_iso_class_eq_dim_h1}
Let $\Gamma$ be a potentially taut dual graph. Let $\OS{Z}$ be an anti-ample divisor for $\Gamma$ with $\gcd(p,n_l)=1$ for all $l$.
Further let $\nu$ be its significant multiplicity and $P$ the plumbing scheme for $\nu \OS{Z}$.
Then we have exactly $1+h^1(P,\Theta_{P})$ isomorphism classes of $\Gamma$-singularities.
\end{conj}
In particular we could reformulate Theorem \ref{thm_0_taut_implies_most_p_taut} as ``$S_1$ is taut if and only if
$S_p$ is taut for all but finitely many $p$''.
Also for the bad $p$ of Theorem \ref{thm_0_taut_implies_most_p_taut} we would have that $S_p$ can not be taut.

\bibliographystyle{../../shared/halpha}	
\bibliography{../../shared/references}


\end{document}